%% file: abt-ssb.tex
\documentclass[11pt,letterpaper]{article}

\usepackage[cmyk]{xcolor}

\usepackage[utf8]{inputenc}
\usepackage[vcentermath]{youngtab}
\usepackage{algorithm}
\usepackage{algpseudocode}
\usepackage[section]{placeins}
    \usepackage{amsmath,amsthm,amssymb,xcolor,tikz,tipa,tgtermes,relsize,appendix,ifthen,xstring,subcaption,array,float,ytableau,pdflscape,setspace,listings}

\usetikzlibrary{decorations.pathmorphing}
\usetikzlibrary{snakes}

\usepackage{hyperref}

\input{preamble}

\title{Approval Ballot Triangles and Strict-Sense Ballots}
\author{Andrew Beveridge\footnote{Department of Mathematics, Statistics and Computer Science, Macalester College, 1600 Grand Avenue, St Paul, Minnesota, 55105, U.S.A. \texttt{abeverid@macalester.edu}} ~and Ian Calaway\footnote{Department of Economics, DePaul University, 1 E. Jackson Blvd, Chicago, Illinois, 60604, U.S.A. \texttt{icalaway@depaul.edu}}}
\date{}
	
\begin{document}

\maketitle

\begin{abstract}
We consider a family of binary triangular arrays, called \emph{approval ballot triangles} (ABTs), that are in bijection with totally symmetric self-complementary plane partitions (TSSCPPs). These triangles correspond to a  ballot process in which voters select their collection of approved candidates rather than voting for a single person.  We situate ABTs within the ballot problem literature and then show that a strict-sense ballot can be decomposed into a list of sequentially compatible ABTs.
\end{abstract}




\section{Introduction}

Inspired by a generalization of Bertrand's ballot problem \cite{bertrand}, we investigate  the following family of binary triangles. 
\begin{definition}
An \emph{approval ballot triangle} (ABT)  of size $n$ is a binary triangular array $A(i,j)$ for $1 \leq j \leq i \leq n$ satisfying the row compatibility condition
\begin{equation}
\label{eqn:abt-row}
\sum_{k=j}^i A(i,k) \leq \sum_{k=j}^{i+1} A(i+1,k) \quad  \mbox{for} \quad 1 \leq j \leq i \leq n-1.
\end{equation}
 We use $\mathcal{A}_n$ to denote the family of approval ballot triangles of size $n$.
\end{definition}
We use matrix indexing for our triangular arrays: rows are indexed top-to-bottom and columns are indexed left-to-right. The forty-two ABTs of $\abt{3}$ are shown in Figure \ref{fig:abt4}.
We rephrase the row compatibility constraint more intuitively: 
when considering column $j$ and above,  there are at least as many ones in row $i+1$ as in row $i$ of an ABT.

Approval ballot triangles model a voting procedure in which each voter chooses an ``approved'' subset of the available candidates. 
We describe this procedure in Section \ref{sec:ballot}, where we also situate ABTs within the ballot problem literature. 
Standard ballot problems have a natural bijection to lattice paths, and we will show in Section \ref{sec:abt} that ABTs biject to collections of non-crossing lattice paths.
More importantly, we show that the family of ABT triangles biject to a famous combinatorial family:  \emph{totally symmetric self-complementary plane partitions} (TSSCPPs), whose definition can be found in Section \ref{sec:tsscpp}. 

\ytableausetup{smalltableaux}

\begin{figure}[ht]
\centering
\begin{tikzpicture}[scale=.67]

\node (v000000) at (11,0) {\scriptsize $\begin{ytableau} ~ \\ ~ & ~ \\ ~ & ~ & ~ \\ \end{ytableau}$ };

\node (v000100) at (9,2) {\scriptsize $\begin{ytableau} ~ \\ ~ & ~ \\ 1 & ~ & ~ \\ \end{ytableau}$ };
\node (v000010) at (11,2) {\scriptsize $\begin{ytableau} ~ \\ ~ & ~ \\ ~ & 1 & ~ \\ \end{ytableau}$ };
\node (v000001) at (13,2) {\scriptsize $\begin{ytableau} ~ \\ ~ & ~ \\ ~ & ~ & 1 \\ \end{ytableau}$ };

\node (v010100) at (4,4) {\scriptsize $\begin{ytableau} ~ \\ 1 & ~ \\ 1 & ~ & ~ \\ \end{ytableau}$ };

\node (v000110) at (6,4) {\scriptsize $\begin{ytableau} ~ \\ ~ & ~ \\ 1 & 1 & ~ \\ \end{ytableau}$ };

\node (v000101) at (10,4) {\scriptsize $\begin{ytableau} ~ \\ ~ & ~ \\ 1 & ~ & 1 \\ \end{ytableau}$ };

\node (v010010) at (8,4) {\scriptsize $\begin{ytableau} ~ \\ 1 & ~ \\ ~ & 1 & ~ \\ \end{ytableau}$ };

\node (v001010) at (14,4) {\scriptsize $\begin{ytableau} ~ \\ ~ & 1 \\ ~ & 1 & ~ \\ \end{ytableau}$ };

\node (v000011) at (16,4) {\scriptsize $\begin{ytableau} ~ \\ ~ & ~ \\ ~ & 1 & 1 \\ \end{ytableau}$ };

\node (v001001) at (18,4) {\scriptsize $\begin{ytableau} ~ \\ ~ & 1 \\ ~ & ~ & 1 \\ \end{ytableau}$ };

\node (v010001) at (12,4) {\scriptsize $\begin{ytableau} ~ \\ 1 & ~ \\ ~ & ~ & 1 \\ \end{ytableau}$ };

\node (v110100) at (0,6) {\scriptsize $\begin{ytableau} 1 \\ 1 & ~ \\ 1 & ~ & ~ \\ \end{ytableau}$ };

\node (v110010) at (2,6) {\scriptsize $\begin{ytableau} 1 \\ 1 & ~ \\ ~ & 1 & ~ \\ \end{ytableau}$ };

\node (v001110) at (8,6) {\scriptsize $\begin{ytableau} ~ \\ ~ & 1 \\ 1 & 1 & ~ \\ \end{ytableau}$ };

\node (v110001) at (12,6) {\scriptsize $\begin{ytableau} 1 \\ 1 & ~ \\ ~ & ~ & 1 \\ \end{ytableau}$ };

\node (v000111) at (10,6) {\scriptsize $\begin{ytableau} ~ \\ ~ & ~ \\ 1 & 1 & 1 \\ \end{ytableau}$ };

\node (v010110) at (4,6) {\scriptsize $\begin{ytableau} ~ \\ 1 & ~ \\ 1 & 1 & ~ \\ \end{ytableau}$ };
\node (v010101) at (6,6) {\scriptsize $\begin{ytableau} ~ \\ 1 & ~ \\ 1 & ~ & 1 \\ \end{ytableau}$ };

\node (v001101) at (16,6) {\scriptsize $\begin{ytableau} ~ \\ ~ & 1 \\ 1 & ~ & 1 \\ \end{ytableau}$ };

\node (v010011) at (14,6) {\scriptsize $\begin{ytableau} ~ \\ 1 & ~ \\ ~ & 1 & 1 \\ \end{ytableau}$ };

\node (v001011) at (18,6) {\scriptsize $\begin{ytableau} ~ \\ ~ & 1 \\ ~ & 1 & 1 \\ \end{ytableau}$ };

\node (v101010) at (20,6) {\scriptsize $\begin{ytableau} 1 \\ ~ & 1 \\ ~ & 1 & ~ \\ \end{ytableau}$ };

\node (v101001) at (22,6) {\scriptsize $\begin{ytableau} 1 \\ ~ & 1 \\ ~ & ~ & 1 \\ \end{ytableau}$ };

\node at (1,8) {\scriptsize $\begin{ytableau} 1 \\ 1 & ~ \\ 1 & 1 & ~ \\ \end{ytableau}$ };

\node at (3,8) {\scriptsize $\begin{ytableau} 1 \\ 1 & ~ \\ 1 & ~ & 1 \\ \end{ytableau}$ };

\node at (5,8) {\scriptsize $\begin{ytableau} ~ \\ 1 & 1 \\ 1 & 1 & ~ \\ \end{ytableau}$ };

\node at (7,8) {\scriptsize $\begin{ytableau} 1 \\ ~ & 1 \\ 1 & 1 & ~ \\ \end{ytableau}$ };

\node at (9,8) {\scriptsize $\begin{ytableau} ~ \\ 1 & ~ \\ 1 & 1 & 1 \\ \end{ytableau}$ };

\node at (13,8) {\scriptsize $\begin{ytableau} ~ \\ ~ & 1 \\ 1 & 1 & 1 \\ \end{ytableau}$ };

\node at (17,8) {\scriptsize $\begin{ytableau} ~ \\ 1 & 1 \\ ~ & 1 & 1 \\ \end{ytableau}$ };

\node at (19,8) {\scriptsize $\begin{ytableau} 1 \\ ~ & 1 \\ 1 & ~ & 1 \\ \end{ytableau}$ };

\node at (21,8) {\scriptsize $\begin{ytableau} 1 \\ ~ & 1 \\ ~ & 1 & 1 \\ \end{ytableau}$ };

\node at (11,8) {\scriptsize $\begin{ytableau} ~ \\ 1 & 1 \\ 1 & ~ & 1 \\ \end{ytableau}$ };

\node at (15,8) {\scriptsize $\begin{ytableau} 1 \\ 1 & ~ \\ ~ & 1 & 1 \\ \end{ytableau}$ };

\node at (12,10) {\scriptsize $\begin{ytableau} ~ \\ 1 & 1 \\ 1 & 1 & 1 \\ \end{ytableau}$ };

\node at (6,10) {\scriptsize $\begin{ytableau} 1 \\ 1 & 1 \\ 1 & 1 & ~ \\ \end{ytableau}$ };

\node at (8,10) {\scriptsize $\begin{ytableau} 1 \\ 1 & ~ \\ 1 & 1 & 1 \\ \end{ytableau}$ };

\node at (14,10) {\scriptsize $\begin{ytableau} 1 \\ ~ & 1 \\ 1 & 1 & 1 \\ \end{ytableau}$ };

\node at (16,10) {\scriptsize $\begin{ytableau} 1 \\ 1 & 1 \\ ~ & 1 & 1 \\ \end{ytableau}$ };

\node at (10,10) {\scriptsize $\begin{ytableau} 1 \\ 1 & 1 \\ 1 & ~ & 1 \\ \end{ytableau}$ };

\node at (11,12) {\scriptsize $\begin{ytableau} 1 \\ 1 & 1 \\ 1 & 1 & 1 \\ \end{ytableau}$ };

\end{tikzpicture}

\caption{The 42 approval ballot triangles of size 3. The zero entries are rendered blank for visual clarity.}
\label{fig:abt4}

\end{figure}
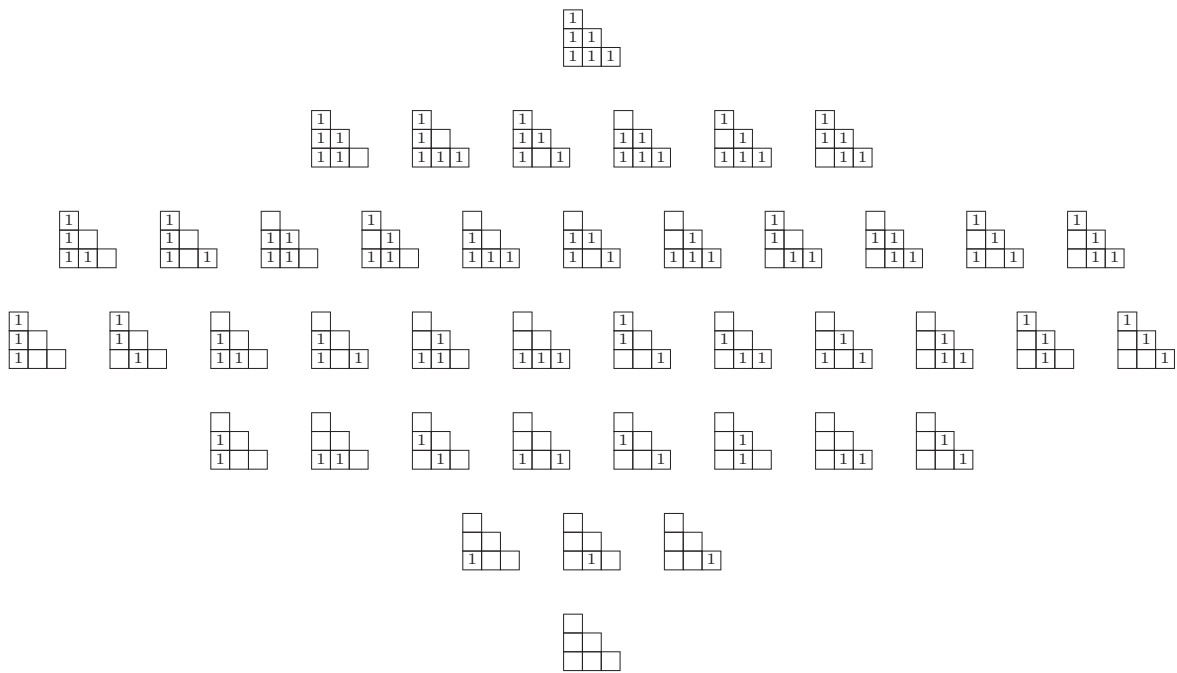

\ytableausetup{nosmalltableaux}

\begin{theorem}
\label{thm:abt-tsscpp}
The set $\mathcal{A}_{n-1}$ of approval ballot triangles of size $n-1$ are in bijection with TSSCPPs inside a $2n \times 2n \times 2n$ box. 
\end{theorem}

Our main result is to reveal the connection between the TSSCPPs and strict-sense ballots.   

\begin{definition}
Consider an election with $n$ candidates in which  candidate $k$ receives exactly $n+1-k$ votes. A \emph{strict-sense ballot} (SSB) is an ordering of these  ${n+1 \choose 2}$ votes  so that candidate $k$  always leads candidate $k+1$ for $1 \leq k \leq n-1$.    
\end{definition}
For example,
\begin{equation}
\label{eqn:ssb-example}
(
1,1,1,2,1,
2,2,3,1,3,
1,2,3,1,2,
4,4,5,2,3,
4,5,3,4,6,
5,6,7)
\end{equation}
is a strict-sense ballot for seven candidates. The number of SSBs for $n$ candidates is
$$
{n+1 \choose 2}! \frac{ \prod_{k=1}^{n-1} k!}{\prod_{k=1}^n (2k-1)!},
$$
see OEIS A003121 \cite{oeis}.
A strict-sense ballot is equivalent to a shifted standard Young tableau (SYT) of staircase shape $(n,n-1,\ldots,1)$  \cite{schur,thrall,hiller}.  We describe the  shifted SYT bijection in Section \ref{sec:ssb}.

We will show that a SSB decomposes into a list of sequentially compatible ABTs (and therefore, also to a list of sequentially compatible TSSCPPs). Before defining ``compatibility'' between our triangular arrays,   we set some notation for subsequences and sub-arrays, and define a partial order on binary sequences.

\begin{definition}
For a sequence $S=(s_1, s_2, \ldots, s_n)$, let $S[i:j]$ denote the subsequence $(s_i, s_{i+1}, \ldots, s_j)$. For a triangular array $T = T(i,j)$ for $1 \leq j \leq i \leq m$, let $T(i, [j:k])$ denote the sub-row $(T(i,j), T(i,j+1), \ldots, T(i,k))$.
\end{definition}

For convenience, we define $S[i,j] = \emptyset$ (the empty sequence) when $j < i$ and $T(i,[j:k]) = \emptyset$ when  either $i < 1$ or  $k < j$.

\begin{definition}
Let $\mathcal{S}$ denote the set of binary sequences (of any length). We define the partial order $(\mathcal{S}, \prec)$ as follows. Let $S,T \in \mathcal{S}$ where   $S=(s_1, s_2, \ldots, s_m)$ and $T=(t_1, t_2, \ldots, t_n)$ Then  $S \prec T$ when  $m \leq n$ and
\begin{equation}
\label{eqn:abt-rowXXXX}
\sum_{k=\ell}^m s_k \leq \sum_{k=\ell}^{n} t_k \quad  \mbox{for} \quad 1 \leq \ell \leq m.
\end{equation}
\end{definition}

In other words $S \prec T$, whenever $|S| \leq |T|$ and there are at least as many 1's in $S[\ell:m]$ as in $T[\ell:n]$.
Note that we are  summing the last $m+1-\ell$ entries of $S$ and the last $n+1-\ell$ entries of $T$. 
 An intuitive way to navigate $m<n$ is by appending $n-m$ zeros to $S$ before checking compatibility. For example, we have
$$
(1,0,0,1,1) \prec (0,1,0,0,1,0,0,1) \Longleftrightarrow
(1,0,0,1,1,\underbrace{0,0,0}_{\mbox{\scriptsize{appended}}}) \prec (0,1,0,0,1,0,0,1)
$$
where we have appended three zeros to $(1,0,0,1,1)$ to obtain sequences of equal lengths. In the case that $m=n$, a simple way to compare $S=(1,0,0,1,1,0,0,0)$ and $T = (0,1,0,0,1,0,0,1)$ is to look at their reverse partial sums 
$$
\hat{S} = (3,2,2,2,1,0,0,0) \quad \mbox{and} \quad \hat{T}=(3,3,2,2,2,1,1,1),
$$
and observe that $\hat{S}(i) \leq \hat{T}(i)$ for all $i$. Finally, we 
note that $S \prec T$ and $T \prec S$ if and only if $m=n$ and $S = T$. Therefore, we use the symbol $\preceq$ when appropriate.

Our first example of this poset notation in action is to reformulate the ABT row compatibility condition \eqref{eqn:abt-row}. For  $A \in \mathcal{A}_n$, we now have the simpler formulation 
$$
A(i,[1:i]) \prec A(i+1,[1:i+1]) \quad \mbox{for} \quad 1 \leq i \leq n-1.
$$
which says that row $i$ precedes row $i+1$ in $(S, \prec)$.  This  poset notation allows us to define our list of compatible approval ballot triangles in an efficient manner.

\begin{definition}
\label{def:abt-hyper}
An  \emph{approval ballot hypertriangle} (ABH) of size $n$ is a sequence $(A_{n}, A_{n-1} \ldots, A_1)$ where $A_{s} \in  \mathcal{A}_{s}$ and such that
whenever $A_s(t,u)=1$ where $2 \leq u \leq t \leq s$, we have the following compatibility conditions on row $t$ of $A_s$. Setting 
$$
s' = s - \sum_{j=u}^t A(t,j) \quad \mbox{and} \quad t'=u-1,
$$
the ABH must satisfy the triangle compatibility condition
\begin{equation}
\label{eqn:abh-compatibility}
 A_{s'} (t'-1,[1:t'-1]) \prec A_s (t,[1:t']) \preceq A_{s'}(t',[1:t'])   
\end{equation}
where the first inequality holds when $1 < t' \leq s'$ and the second holds when $1 \leq t' \leq s'$.
\end{definition}

Figure \ref{fig:abh-example} shows an approval ballot hypertriangle $A=(A_5,A_3,A_2,A_2,A_1)$ of size 5. In Section \ref{sec:ballot}, we explain the geometric meaning of the ABH triangle compatibility conditions. Recalling that each ABT corresponds to a set of non-crossing lattice paths, we will see that condition \eqref{eqn:abh-compatibility} ensures that paths from different ABTs are also pairwise non-crossing. (The non-crossing lattice paths for the ABH of Figure \ref{fig:abh-example} are shown in Figure \ref{fig:abt-hyper} below.)

\ytableausetup{smalltableaux}
\begin{figure}[ht]
    \centering
\begin{tikzpicture}

\node at (-4,0) {$(A_5,A_4,A_3,A_2,A_1)$};

\node at (-1.75,0) {$=$};

\node at (0,0) {
\begin{ytableau}
~ \\
1 & ~ \\
1 & ~ & 1 \\
~ & 1 & 1 & ~  \\
~ & ~ & 1 & 1 & ~
\end{ytableau}
};

\node at (2.4,-.15) {
\begin{ytableau}
~ \\ 
~ & 1 \\
~ & ~ & 1 \\
1 & 1 & 1 & ~ 
\end{ytableau}
};

\node at (4.5,-.3) {
\begin{ytableau}
~ \\ 
1 & 1 \\
~ & 1 & 1 \\
\end{ytableau}
};

\node at (6.3,-.45) {
\begin{ytableau}
1 \\ 
~ & 1 \\
\end{ytableau}
};

\node at (7.8,-.6) {
\begin{ytableau}
1 \\ 
\end{ytableau}
};

\end{tikzpicture}
    \caption{An approval ballot hypertriangle  $(A_5,A_4,A_3,A_2,A_1)$ of size 5.} 
    \label{fig:abh-example}
\end{figure}
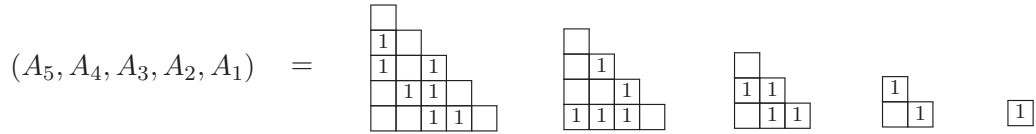

When $A_s(t,u)=1$,  for $2 \leq u \leq t$, there are two row constraints between the triangle $A_s$ and the smaller triangle $A_{s'}$. Figure \ref{fig:abh-row-constraint} shows a schematic example of the ABH constraints induced by a row of  triangle $A_{9}$ in the ABH $(A_{9}, A_8, \ldots, A_1)$. Subrow $A_{9}(7,[2:7])$ contains three ones, located at indices 7, 4 and 3. This leads to the constraints
$$
\begin{array}{cccccc}
A_8(5,[1:5]) &\prec& A_{9}(7,[1:6]) &\preceq& A_8(6,[1:6]), \\
A_7(2,[1:2]) &\prec& A_{9}(7,[1:3]) &\preceq& A_7(3,[1:3]), \\
A_6(1,[1:1]) &\prec& A_{9}(7,[1:2]) &\preceq& A_6(6,[1:2]). \\
\end{array}
$$
For the second of these constraints, we have
$A_{9}(7,[1:3]) = (1,0,1)$ and therefore
 $A_7(2,[1:2]) \in \{ (0,0), (0,1), (1,0), (1,1) \}$ and
$A_7(3,[1:3]) \in \{ (1,0,1), (0,1,1), (1,1,1) \}$.

\ytableausetup{smalltableaux}
\begin{figure}[ht]
    \centering
\begin{tikzpicture}


\node at (0,-2) {\scriptsize $(s,t)=(10,7)$};

\node at (0.65,0.65) {\small $A_{9}$};

\node at (0,0) {
\begin{ytableau}
~ \\
~ & ~ \\
~ & ~ & ~ \\
~ & ~ & ~ & ~  \\
~ & ~ & ~ & ~  & ~ \\
~ & ~ & ~ & ~  &  ~ & ~ \\
 1 &0 &  *(gray!25) 1 & *(gray!50)1 & 0 & 0 & *(gray!75) 1 \\
~ & ~ & ~ & ~  &  ~ & ~ & ~ & ~\\
~ & ~ & ~ & ~  &  ~ & ~ & ~ & ~ & ~\\
\end{ytableau}
};

\begin{scope}[shift={(3.5,-0.15)}]

\node at (.05,-1.65+.15) {\scriptsize $u=7$};
\node at (0,-1.65-.35) {\scriptsize $(s',t')=(9,6)$};

\node at (0.65,0.8) {\small $A_{8}$};

\end{scope}

\begin{scope}[shift={(3.5,0.15)}]

\node at (0,0) {
\begin{ytableau}
~ \\
~ & ~ \\
~ & ~ & ~ \\ 
~ & ~ & ~ & ~  \\
*(gray!75) ~ & *(gray!75) ~ & *(gray!75) ~ & *(gray!75) ~  & *(gray!75) ~  \\
*(gray!75) ~ & *(gray!75) ~ & *(gray!75) ~ &  *(gray!75)~  &   *(gray!75) ~ &  *(gray!75)~ \\
~ & ~ & ~ & ~  &  ~ & ~ & ~ \\
~ & ~ & ~ & ~  &  ~ & ~ & ~ & ~\\
\end{ytableau}
};

\end{scope}

\begin{scope}[shift={(6.7,-0.3)}]

\node at (.05,-1.65+.3) {\scriptsize $u=4$};
\node at (0,-1.65-.2) {\scriptsize $(s',t')=(8,3)$};

\node at (0.65,0.95) {\small $A_{7}$};

\end{scope}

\begin{scope}[shift={(6.7,0.3)}]

\node at (0,0) {
\begin{ytableau}
~ \\
*(gray!50) ~ & *(gray!50) ~ \\
*(gray!50) ~ & *(gray!50)  ~ & *(gray!50)  ~ \\
~ & ~ & ~ & ~  \\
~ & ~ & ~ & ~  & ~ \\
~ & ~ & ~ & ~  &  ~ & ~ \\
~ & ~ & ~ & ~  &  ~ & ~ & ~ \\
\end{ytableau}
};

\end{scope}

\begin{scope}[shift={(9.55,-0.45)}]

\node at (.05,-1.65+.45) {\scriptsize $u=3$};
\node at (0,-1.65-.05) {\scriptsize $(s',t')=(7,2)$};

\node at (0.65,1.10) {\small $A_{6}$};

\end{scope}

\begin{scope}[shift={(9.55,0.45)}]

\node at (0,0) {
\begin{ytableau}
*(gray!25) ~  \\
*(gray!25) ~ & *(gray!25)   ~ \\
~ & ~ & ~  \\
~ & ~ & ~ & ~  \\
~ & ~ & ~ & ~  & ~\\
~ & ~ & ~ & ~  &  ~ & ~\\
\end{ytableau}
};

\end{scope}

\draw[fill] (11.5,-.1) circle (1pt); 
\draw[fill] (11.75,-.1) circle (1pt); 
\draw[fill] (12,-.1) circle (1pt); 

\end{tikzpicture}

\caption{An example of the ABH triangle constraint of Definition \ref{def:abt-hyper}. We have an ABH $(A_{9}, A_{8}, A_7, A_6, \ldots)$. The three ones in $A_{9}(7,[2:7])$ are located at $(t,u) \in \{(7,3), (7,4), (7,7)\}$. Each of these ones induces two constraints on one of the next three triangles in the ABH. Each one in row 7 of $A_9$ has the same shading as the two rows affected by the corresponding constraint.}

\label{fig:abh-row-constraint}

\end{figure}
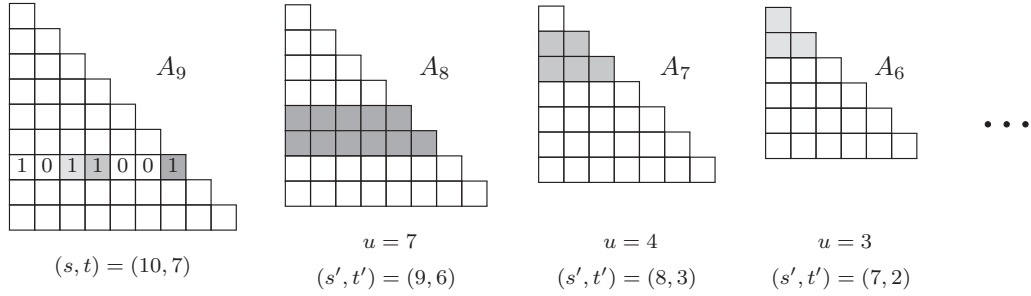

This brings us to our main theorem.

\begin{theorem}
\label{thm:hypertriangle}
Approval ballot hypertriangles $(A_{n-2}, A_{n-3}, \ldots, A_1)$ of size $n-2$  are in bijection with strict-sense ballots for $n$ candidates.
\end{theorem}

It is well-known that a TSSCPP can be viewed as a list of sequentially compatible Dyck paths of decreasing sizes \cite{striker2011}. Combined with Theorem \ref{thm:abt-tsscpp}, Theorem  \ref{thm:hypertriangle} states that a strict-sense ballot can be viewed as a list of sequentially compatible TSSCPPs of decreasing sizes. 
This constructive phenomenon is one of the natural laws of combinatorics: careful aggregation of simple structures gives rise to other interesting structures.

The paper is organized as follows. In Section \ref{sec:ballot}, we establish ABTs within the realm of the ballot problems. 
Section \ref{sec:abt} starts with a discussion of ABTs and nests of lattice paths, followed by the proof of Theorem \ref{thm:abt-tsscpp}. 
 We discuss strict-sense ballots in Section \ref{sec:ssb}, where we prove Theorem \ref{thm:hypertriangle}.  We reflect on our results in Section \ref{sec:end}.

\section{Ballot Problems}
\label{sec:ballot}

In this section, we situate ABTs  within the domain of ballot problems.
See Barton and Mallows \cite{barton}, Tak{\`a}cs \cite{takacs}, Renault \cite{renault}, and Addario-Berry and Reed \cite{addario} for ballot problem surveys.

Consider an election between candidates $A$ and $B$, receiving $a$ and $b$ respective votes with $a > b$. Bertrand  \cite{bertrand} showed that there are 
$
\frac{a-b}{a+b} {a+b \choose a}
$
distinct orderings of the $a+b$ ballots in which candidate $A$ is always ahead of candidate $B$. 
These orderings are in bijection with lattice paths from $(0,0)$ to $(a+b,a-b)$ using up $(1,1)$ steps and down $(1,-1)$ steps and staying strictly above the $x$-axis.
If we relax the constraint to allow for ties during the partial vote count, the number of orderings is   
$
\frac{a+1-b}{a+1} {a+b \choose a}.
$
The corresponding lattice paths are now allowed to touch the $x$-axis.
In the case of a tied election $a=b=n$, the solution to the ``$A$ never trails $B$'' ballot problem is the $n$th Catalan number
$
C_n = \frac{1}{n+1}{2n \choose n},
$
see OEIS A000108 \cite{oeis}.
The corresponding lattice paths, from $(0,0)$ to $(2n,0)$, using up $(1,1)$ steps and down $(1,-1)$ steps, that do not travel below the $x$-axis, are called \emph{Dyck paths}.

Next, consider an electorate $[n]$ where every voter is also an eligible candidate. Examples include a club electing a leader from among its members, and an assembly nominating members for a committee.

\begin{definition}
\label{def:ballot-seq}
The sequence  $b_1, \ldots, b_n$ where $b_k \in [k]$ is a \emph{ballot sequence of length $n$} when for every $1 \leq s \leq n$ and $1 \leq t  \leq n-1$, the partial sequence $b_1, \ldots, b_s$ contains at least as many $t$'s  as $(t+1)$'s. In particular the final tally for  $t$ is greater than or equal to the final tally for  $t+1$.
\end{definition}

In a ballot sequence, candidate $t$ never trails candidate $(t+1)$ as the votes are counted. One particular consequence is that $b_i \in [i]$, meaning that person $i$ cannot vote for person $j > i$. Here is an elementary voting procedure that obeys this rule: people enter the room one at a time to cast their vote for someone who is already present (including themselves). The votes are then revealed in the order that they were cast. We are interested in the sequences that adhere to the ``$t$ never trails $t+1$'' constraint.

We are now ready to describe the ballot process that corresponds to approval ballot triangles. 
Just as for a ballot sequence, voter $k$ can only vote for candidates in $[k]$. 
Instead of using the standard voting process, we will use \emph{approval voting} \cite{brams+fishburn-1978, brams+fishburn-2005}. Each voter specifies their subset of  ``approved'' candidates (rather than selecting a single preferred candidate). Each of these approved candidates receives one vote in their favor. The winner of the election is the candidate who receives the most approval votes. For some mathematical investigations of approval voting, see \cite{berg2010voting,su+zerbib}.

\begin{definition}
Let $B_1, B_2, \ldots, B_n$ be a sequence of sets   $B_k \subset [k]$. This sequence is an \emph{approval ballot sequence} when for every for every $1 \leq s \leq n$ and $1 \leq t  \leq n-1$, the partial set sequence $B_1, B_2, \ldots, B_s$ contains at least as many $t$'s as $(t+1)$'s. 
\end{definition}

Note that $B_k = \emptyset$  means that  the $k$th voter abstained. At the other extreme, only the $n$th voter can approve of every candidate because $B_k \subset [k]$.

\begin{prop}
Approval ballot triangles of size $n$ are in bijection with approval ballot sequences of length $n$.
\end{prop}

\begin{proof}
Consider the approval ballot sequence $B_1, B_2, \ldots, B_{n}$.
Since $j \notin B_i$ for $j > i$, we can write the approval ballot sequence condition as
$$
\sum_{u=t+1}^s \mathbf{1}_{t+1}(B_u) \leq 
\sum_{u=t}^s \mathbf{1}_{t}(B_u). 
$$
for $1 \leq t \leq s \leq n-1$, where $\mathbf{1}_j(S)$ is the indicator function for $j \in S$.

The row and column indexing for approval ballot triangle $A$ is the reverse of the approval ballot sequence $B_1, B_2, \ldots, B_{n}$. For $1 \leq s \leq u \leq n-1$, set $A(t,u) = \mathbf{1}_{n+1-t}(B_{n+1-r})$, which is equivalent to $\mathbf{1}_{t}(B_{u}) = A(n+1-t, n+1-u)$. The approval ballot sequence condition becomes
\begin{align*}
\sum_{u=t+1}^s A(n-t, n+1-u) &\leq 
\sum_{u=t}^s A(n+1-t, n+1-u)
\\
\sum_{k=j}^{i} A(i, k) &\leq 
\sum_{k=j}^{i+1} A(i+1, k)
\end{align*}
where we use the change of variables $k=n+1-r$ and $i=n-t$ and $j=n+1-s$. 
\end{proof}

\begin{remark}
Given the re-indexing complexity of the previous proof, it is natural to wonder why we have adopted the ``reverse indexing'' for ABTs. The reason is simple: our definition of ABTs leads to the very simple ABH compatibility equation \eqref{eqn:abh-compatibility}. Other (seemingly natural) choices for ABT organization result in far more complicated expressions of this compatibility constraint.
\end{remark}

We conclude this section by observing how ABTs generalize Dyck paths and Motzkin paths (defined below). Consider  an approval election between candidates $A$ and $B$,  who both receive $k$ votes. If no voter abstains and  no voter approved of both $A$ and $B$, then we have a standard election  where $2k$ ballots were cast. Therefore the number of ways to count the votes so that $A$ never trails $B$  is the $k$th Catalan number $C_k$.
The corresponding ABTs have $2k$ columns. There is a unique one in each column, and these ones are restricted to the last two rows, each of which contains $k$ ones. The last two rows of the ABTs for $2k=6$ ballots are shown in Figure \ref{fig:abt-catalan}(a). Read from right-to-left, these encode Dyck paths of length $6$.

Next, consider a $n$-voter election  between candidates $A$ and $B$  where abstentions are allowed, but not approval of both candidates. If $A$ and $B$ receive an equal number $0 \leq k \leq \lfloor n/2 \rfloor$ of votes, then the number of ways to count the votes so that $A$ never trails $B$  is the $n$th Motzkin number $M_n = \sum_{k=0}^{\lfloor n/2 \rfloor} {n \choose 2k} C_k$, see OEIS A001006 \cite{oeis}. 
The corresponding lattice paths from $(0,0)$ to $(n,0)$ using up $(1,1)$ steps, down $(1,-1)$ steps and horizontal $(1,0)$ steps, while never moving below the $x$-axis, are called \emph{Motzkin paths}. These three step types correspond to $A$ votes, $B$ votes and abstentions, respectively. The ABTs for Motzkin paths have $n$ columns, at most one 1 in each column, and these ones are restricted to the last two rows, with an equal number of ones these rows. Figure \ref{fig:abt-catalan}(b) shows the last two rows of the ABTs for $n=4$ that encode the Motzkin paths for these elections, read right-to-left.

\ytableausetup{smalltableaux}

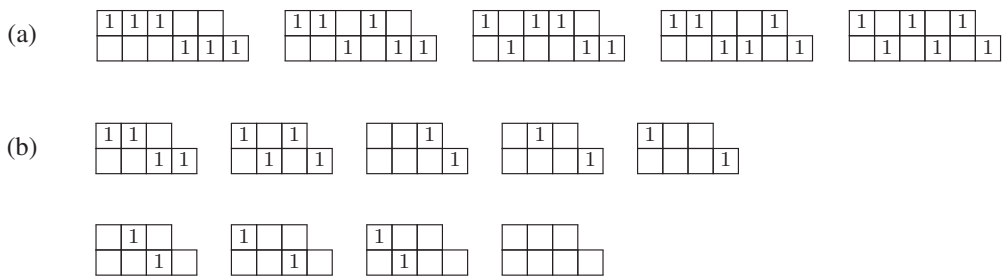
\begin{figure}[ht]

\begin{center}
\begin{tikzpicture}

\node at (-2,0) {\small (a)};

\node at (0,0) {\begin{ytableau} 
1 & 1 & 1 & ~ & ~ \\
~ & ~ &  ~ & 1 & 1 & 1
\end{ytableau}};

\node at (2.5,0) {\begin{ytableau} 
1 & 1 & ~ & 1 & ~ \\
~ & ~ &  1 & ~ & 1 & 1 
\end{ytableau}};

\node at (5,0) {\begin{ytableau} 
1 & ~ & 1 & 1 & ~ \\
~ & 1 &  ~ & ~ & 1 & 1 
\end{ytableau}};

\node at (7.5,0) {\begin{ytableau} 
1 & 1 & ~ & ~ & 1 \\
~ & ~ &  1 & 1 & ~ & 1 
\end{ytableau}};

\node at (10,0) {\begin{ytableau} 
1 & ~ & 1 & ~ & 1 \\
~ & 1 &  ~ & 1 & ~ & 1 
\end{ytableau}};

\node at (-2,-1.5) {\small (b)};

\begin{scope}[shift={(-.35,-1.5)}, scale=.9]

\node at (2,0) {\begin{ytableau} 
 1 & ~ & 1 \\
  ~ & 1 & ~ & 1
\end{ytableau}};

\node at (0,0) {\begin{ytableau} 
 1 & 1 & ~ \\
  ~ & ~ & 1 & 1
\end{ytableau}};

\node at (4,0) {\begin{ytableau} 
 ~ & ~ & 1 \\
  ~ & ~ & ~ & 1
\end{ytableau}};

\node at (6,0) {\begin{ytableau} 
 ~ & 1 & ~ \\
  ~ & ~ & ~ & 1
\end{ytableau}};

\node at (8,0) {\begin{ytableau} 
 1 & ~ & ~ \\
  ~ & ~ & ~ & 1
\end{ytableau}};

\node at (0,-1.5) {\begin{ytableau} 
 ~ & 1 & ~ \\
  ~ & ~ & 1 & ~
\end{ytableau}};

\node at (2,-1.5) {\begin{ytableau} 
 1 & ~ & ~ \\
  ~ & ~ & 1 & ~
\end{ytableau}};

\node at (4,-1.5) {\begin{ytableau} 
 1 & ~ & ~ \\
  ~ & 1 & ~ & ~
\end{ytableau}};

\node at (6,-1.5) {\begin{ytableau} 
 ~ & ~ & ~ \\
  ~ & ~ & ~ & ~
\end{ytableau}};

\end{scope}

\end{tikzpicture}
\end{center}

\caption{(a) The last two rows of the five ABTs corresponding to Dyck paths of length 6. (b) The last two rows of the nine ABTs corresponding to Motzkin paths of length 4. }

\label{fig:abt-catalan}

\end{figure}

\ytableausetup{nosmalltableaux}


\section{Approval Ballot Triangles}

\label{sec:abt}

The primary goal of this section is to prove Theorem \ref{thm:abt-tsscpp}. Section \ref{sec:tsscpp} offers a quick introduction to totally symmetric self-complementary plane partitions. In Section \ref{sec:abt-lattice-paths}, we represent a TSSCPP as a nest of non-crossing lattice paths. We then prove Theorem \ref{thm:abt-tsscpp}.

\subsection{Totally Symmetric Self-Complementary Plane Partitions}

\label{sec:tsscpp}

A \emph{plane partition} $\pi = \pi(i,j)$ is a two-dimensional array of nonnegative integers with weakly decreasing rows and columns. Typically, we represent these values as stacks of cubes in $\mathbb{R}^3$. A \emph{totally symmetric self-complementary plane partition} (TSSCPP) in a $2n \times 2n \times 2n$ box is a  $2n \times 2n$ plane partition whose cube stack representation has the maximum possible symmetry. More precisely, it is invariant under permutation of its three axes, and it is equal to its complement (which is the set of missing cubes in the $2n \times 2n \times 2n$ box). See Figure \ref{fig:tsscpp}(a) for an example where $n=5$. 
Andrews \cite{andrews} proved that the number of TSSCPP of order $n$  is 
$$
\prod_{k=0}^{n-1} \frac{(3k+1)!}{(n+k)!},
$$
see OEIS A005130 \cite{oeis}.
Bressoud \cite{bressoud} recounts the history of this formula, which sits at the confluence of many combinatorial families.
Notably, this  formula also enumerates descending plane partitions (DPPs) and alternating sign matrices (ASMs). Andrews \cite{andrews1979} had already proven this formula for descending plane partitions. Establishing this result for alternating sign matrices was  first achieved by Zeilberger \cite{zeilberger} and then by Kuperberg \cite{kuperberg} with a simpler argument drawing on the 6-vertex model from statistical mechanics. Recently, Fischer and Konvalinka \cite{fischer1,fischer2} established a landmark bijection between DPPs and ASMs.

Thanks to their symmetries, a TSSCPP $\pi$ is completely determined by its \emph{fundamental domain}, which is the triangular  subarray for indices $n+1 \leq i \leq j \leq 2n$. It is then convenient to change coordinates, and also standard to add one to each entry to obtain a positive triangular array $M(i,j) = 1 +  \pi(2n+1-i, 2n+1-j)$ for $1 \leq j \leq i \leq n$, called a \emph{magog triangle}. A magog triangle can be defined directly as  a triangular array such that  $1 \leq M(i,j) \leq j$ and with  weakly increasing rows and weakly increasing columns.
Figure \ref{fig:tsscpp} shows a  TSSCPP of order 5 and its associated fundamental region and magog triangle.

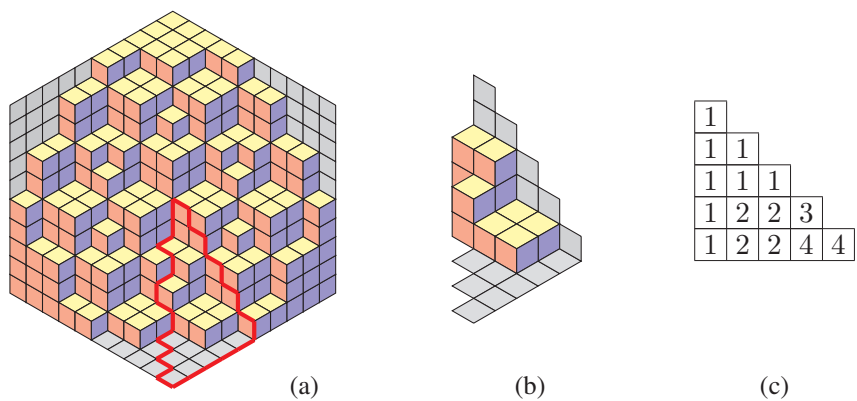
\begin{figure}[ht]

\begin{center}

 \begin{tikzpicture}

 \begin{scope}[scale=.25]

 \planepartition{
 {10,10,10,10,10,9,9,7,7,5},
 {10,10,10,9,9,7,7,5,5,3},
 {10,10,10,9,9,7,6,5,5,3},
 {10,9,9,8,7,5,5,4,3,1},
 {10,9,9,7,7,5,5,3,3,1},
 {9,7,7,5,5,3,3,1,1,0},
 {9,7,6,5,5,3,2,1,1,0},
 {7,5,5,4,3,1,1,0,0,0},
 {7,5,5,3,3,1,1,0,0,0},
 {5,3,3,1,1,0,0,0,0,0},
 }

 \topsidex{3}{-1}{-7}

 \topsidex{3}{0}{-7}

 \topsidex{1}{1}{-7}
 \topsidex{2}{1}{-7}
 \topsidex{3}{1}{-7}

 \topsidex{1}{2}{-7}
 \topsidex{2}{2}{-7}
 \topsidex{3}{2}{-7}

 \topsidex{-1}{3}{-7}
 \topsidex{0}{3}{-7}
 \topsidex{1}{3}{-7}
 \topsidex{2}{3}{-7}
 \topsidex{3}{3}{-7}

 \leftsidex{-9}{1}{-3}

 \leftsidex{-9}{1}{-2}

 \leftsidex{-9}{-1}{-1}
 \leftsidex{-9}{0}{-1}
 \leftsidex{-9}{1}{-1}

 \leftsidex{-9}{-1}{0}
 \leftsidex{-9}{0}{0}
 \leftsidex{-9}{1}{0}

 \leftsidex{-9}{-3}{1}
 \leftsidex{-9}{-2}{1}
 \leftsidex{-9}{-1}{1}
 \leftsidex{-9}{0}{1}
 \leftsidex{-9}{1}{1}

 \rightsidex{1}{-9}{1}
 \rightsidex{0}{-9}{1}
 \rightsidex{-1}{-9}{1}
 \rightsidex{-2}{-9}{1}
 \rightsidex{-3}{-9}{1}

 \rightsidex{1}{-9}{0}
 \rightsidex{0}{-9}{0}
 \rightsidex{-1}{-9}{0}

 \rightsidex{1}{-9}{-1}
 \rightsidex{0}{-9}{-1}
 \rightsidex{-1}{-9}{-1}

 \rightsidex{1}{-9}{-2}

 \rightsidex{1}{-9}{-3}

 \flatstep{3}{3}{-7}{red}

 \upstep{3}{3}{-7}{red}
 \upstep{2}{3}{-7}{red}
 \upstep{1}{3}{-7}{red}
 \upstep{0}{3}{-7}{red}
 \upstep{-1}{3}{-7}{red}

 \backstep{-2}{3}{-6}{red}
 \backstep{-2}{2}{-5}{red}
 \backstep{-2}{1}{-4}{red}
 \backstep{-2}{0}{-3}{red}
 \backstep{-2}{-1}{-2}{red}

 \flatstep{-2}{3}{-6}{red}
 \flatstep{-2}{2}{-5}{red}
 \flatstep{-2}{1}{-4}{red}
 \flatstep{-2}{0}{-3}{red}
 \flatstep{-2}{-1}{-2}{red}

 \backstep{-1}{-1}{-2}{red}
 \backstep{-1}{-1}{-1}{red}

 \upstep{0}{-1}{-3}{red}
 \flatstep{0}{0}{-3}{red}

 \backstep{0}{0}{-3}{red}

 \upstep{0}{-1}{-5}{red}

 \backstep{0}{-1}{-5}{red}

 \flatstep{0}{0}{-6}{red}

 \backstep{0}{0}{-6}{red}

 \upstep{1}{0}{-7}{red}
 \upstep{1}{0}{-8}{red}
 \upstep{1}{0}{-9}{red}

 \flatstep{0}{0}{-9}{red}
 \flatstep{0}{0}{-8}{red}

 \end{scope}

 \node at (1.75,-2.5) {\small (a)};

 \begin{scope}[shift={(4,0)},scale=.33]

 \planepartition{{3,3,1,1,0},{0,2,1,1,0}}
 \topsidex{0}{0}{-4}
 \topsidex{0}{1}{-4}

 \topsidex{-1}{-1}{-4}
 \topsidex{-1}{0}{-4}
 \topsidex{-1}{1}{-4}

 \topsidex{-2}{1}{-4}
 \topsidex{-3}{1}{-4}
 \topsidex{1}{1}{-4}

 \leftsidex{-5}{0}{-4}
 \leftsidex{-5}{-1}{-3}
 \leftsidex{-5}{-2}{-2}
 \leftsidex{-5}{-3}{-1}
 \leftsidex{-5}{-4}{0}

 \leftsidex{-5}{-2}{-3}
 \leftsidex{-5}{-4}{-1}

 \end{scope}

 \node at (4.75,-2.5) {\small (b)};

 \node at (8,.25) { \young(1,11,111,1223,12244) };

 \node at (8,-2.5) {\small (c)};

 \end{tikzpicture}

\end{center}

\caption{(a) A TSSCPP of order 5 along with its (b) fundamental domain,  and (c) magog triangle.}

\label{fig:tsscpp}

\end{figure}

In summary, a TSSCPP is determined by its fundamental domain, which in turn can be represented as a magog triangle. Therefore we treat these combinatorial families as equivalent. Our mappings below will use a lattice path representation of the fundamental domain of a TSSCPP.

\subsection{Non-crossing Lattice Paths}
\label{sec:abt-lattice-paths}

We use  lattice paths to show that ABTs are in bijection with TSSCPPs.   

Let us consider two equivalent lattice path patterns. The first is the classic non-intersecting family introduced by Doran \cite{doran}. The second is the non-crossing family  which we use  in the sections that follow. Non-intersecting paths cannot touch in any manner, whereas non-crossing paths are allowed to share points or steps. 
An example of each is shown in Figure \ref{fig:lattice-path}.

\begin{definition}[\cite{doran}]
\label{def:nilp} A \emph{nest of 
non-intersecting lattice paths} (NILP) of size $n$ is a sequence of non-intersecting paths $P_1, P_2, \ldots, P_{n}$ where path $P_i$ starts at $(2i,i)$ and ends at a point $(j,0)$ where $i \leq j \leq 2i$, taking only south  steps $(0,-1)$ and southwest steps $(-1, -1)$.
\end{definition}

\begin{prop}[\cite{doran}, Theorem 1.2]
\label{prop:tsscpp-non-intersecting}
TSSCPPs in a $2n  \times 2n \times 2n$ box are in bijection with NILPs of size $n-1$. 
\hfill $\Box$
\end{prop}

Striker \cite{striker2018} introduces a binary family of triangles to encode NILPs. 

\begin{definition}[\cite{striker2018}]
\label{def:tsscpp-boolean}
A TSSCPP Boolean triangle of order $n$ is a triangular binary array $B(r,s)$ for $1 \leq r \leq n-1$ and $n-r \leq s \leq n-1$ which satisfies the column compatibility constraint
\begin{equation}
\label{eqn:tsscpp-boolean}
1+\sum_{r=t+1}^{s}  B(r,n-t-1) \geq \sum_{r=t}^{s} B(r,n-t) \quad \mbox{for} \quad  t+1 \leq s \leq n-1.
\end{equation}
\end{definition}

Note that TSSCPP Boolean triangles of order $n$ have $n-1$ rows.

\begin{prop}[\cite{striker2018}, Proposition 2.13]
\label{prop:tsscpp-boolean}
TSSCPP Boolean triangles of order $n$ are in bijection with NILPs of order $n$, and therefore also in bijection with TSSCPP of order $n$. \hfill $\Box$
\end{prop}

An example of this mapping for a TSSCPP of order $6$ is shown in the right half of Figure \ref{fig:lattice-path}. We slice the fundamental domain into horizontal layers. Tracing the boundary path of each level results in five lattice paths $P_1, P_2, P_3, P_4, P_5$.
The isometric view of these boundary paths is a NILP, rotated by $\pi/3$. These paths are encoded in the columns of the corresponding TSSCPP Boolean triangle.

\ytableausetup{boxsize=1.1em}

\begin{figure}[ht]
    \centering
\begin{tikzpicture}

\begin{scope}[shift={(0,2.5)}, scale=.33]
\planepartitionx{{4,4,3,1,1},{0,2,1,1,1},{0,0,1,1,0}}
\topsidex{0}{0}{-4}
\topsidex{0}{1}{-4}
\topsidex{-1}{1}{-4}

\topsidex{0}{0}{-5}
\topsidex{0}{1}{-5}
\topsidex{-1}{1}{-5}
\topsidex{-2}{1}{-5}
\topsidex{-3}{1}{-5}
\topsidex{-4}{1}{-5}
\topsidex{0}{0}{-6}
\leftsidex{-5}{1}{-4}
\leftsidex{-5}{0}{-3}
\leftsidex{-5}{-1}{-2}
\leftsidex{-5}{-2}{-1}
\leftsidex{-5}{-3}{0}
\leftsidex{-5}{-4}{1}

\leftsidex{-5}{-1}{-3}
\leftsidex{-5}{-4}{0}

\upstep{-3}{1}{-3}{red}
\upstep{-2}{1}{-3}{red}
\flatstep{-2}{1}{-3}{red}
\upstep{-1}{0}{-3}{red}
\flatstep{-1}{0}{-3}{red}

\flatstep{-3}{1}{-1}{orange}
\upstep{-2}{0}{-1}{orange}
\flatstep{-2}{0}{-1}{orange}
\upstep{-1}{-1}{-1}{orange}

\upstep{-3}{-1}{-1}{green}
\flatstep{-3}{-1}{-1}{green}
\flatstep{-2}{-1}{0}{green}

\upstep{-2}{-1}{1}{blue}
\flatstep{-2}{-1}{1}{blue}

\flatstep{-2}{-1}{3}{violet}

\startdot{-4}{1}{-3}{red}
\startdot{-3}{1}{-1}{orange}
\startdot{-4}{-1}{-1}{green}
\startdot{-3}{-1}{1}{blue}
\startdot{-2}{-1}{3}{violet}

\node at (1.5, -7.5) {\small fundamental domain};

\end{scope}

\begin{scope}[shift={(3,4)}, scale=0.4]

\node[above] at (10,5) {\scriptsize $P_1$};
\node[above] at (8,4) {\scriptsize $P_2$};
\node[above] at (6,3) {\scriptsize $P_3$};
\node[above] at (4,2) {\scriptsize $P_4$};
\node[above] at (2,1) {\scriptsize $P_5$};

\draw[red, ultra thick] (10,5) -- (10, 4) -- (10,3) --  (9,2) -- (9,1) -- (8,0);
\draw[orange, ultra thick] (8,4) -- (7,3) --  (7,2) -- (6,1) -- (6,0);
\draw[green, ultra thick] (6,3) --  (6,2) -- (5,1) -- (4,0);
\draw[blue, ultra thick] (4,2) -- (4,1) -- (3,0);
\draw[violet, ultra thick] (2,1) -- (1,0);

\foreach \i in {1, 2, 3, 4, 5 }
{
\foreach \j in  {0 , ..., \i}
{
\fill (2*\i , \j) circle (2 pt);
}
\draw[thick, fill=white] (2*\i , \i) circle (3 pt);
}

\foreach \i in {1, 2, 3, 4, 5 }
{
\foreach \j in  {1 , ..., \i}
{
\fill (2*\i -1 , \j - 1) circle (2 pt);
}
}

\node at (5, -2) {non-intersecting lattice paths};

\end{scope}

\begin{scope}[shift={(-5,4)}, scale=0.4]

\draw[dotted] (0,0) -- (5,0) -- (5,1) -- (4,1) -- (4,2) -- (3,2) -- (3,3) -- (2,3) -- (2,4) -- (1,4) -- (1,5) -- (0,5) -- cycle;
\draw[red, ultra thick] (0.0,0) -- (0.0,2) -- (1,2)  -- (1,3) -- (2,3);
\draw[orange, ultra thick] (1,0.0) -- (1.85,0.0) -- (1.85,1.15) -- (3,1.15) -- (3,2);
\draw[green, ultra thick] (2.0,0) -- (2.0,1) -- (4,1);
\draw[blue, ultra thick] (3,0) -- (3,0.85) -- (4,0.85);
\draw[violet, ultra thick] (4,0.0) -- (5,0.0);

\draw[thick, fill=white] (0,0) circle  (3pt);
\draw[thick, fill=white] (1.0,0) circle  (3pt);
\draw[thick, fill=white] (2,0) circle  (3pt);
\draw[thick, fill=white]  (3,0) circle  (3pt);
\draw[thick, fill=white]  (4.0,0) circle  (3pt);

\node[below] at (-.4,0) {\scriptsize $Q_5$};
\node[below] at (.8,0) {\scriptsize $Q_4$};
\node[below] at (2,0) {\scriptsize $Q_3$};
\node[below] at (3.2,0) {\scriptsize $Q_2$};
\node[below] at (4.4,0) {\scriptsize $Q_1$};

\fill (0, 0) circle (2pt);

\foreach \x in {0,1, ..., 5} {
\foreach \y in {0, ..., \x} {
    \fill (5- \x,  \y) circle (2pt);
}
}

\node at (3, -3) {non-crossing lattice paths};

\end{scope}

\begin{scope}[shift={(-4,0)}]
\node  (abt) at (0,0) {\young(0,01,001,1010,01011)};
\node  at (0,-2) {approval ballot triangle};

\node[above] at (-.35,.7) {\scriptsize $Q_1$};
\node[above] at (0.05,.25) {\scriptsize $Q_2$};
\node[above] at (.5,-.2) {\scriptsize $Q_3$};
\node[above] at (0.95,-.65) {\scriptsize $Q_4$};
\node[above] at (1.4,-1.05) {\scriptsize $Q_5$};

\end{scope}

\begin{scope}[shift={(5,0)}]
\node (tsscppbool) at (0,0) {$\begin{ytableau}
\none & \none & \none & \none & 0 \\
\none & \none & \none & 1& 0 \\
\none & \none & 0 & 0 & 1 \\
\none & 0 & 1 & 1 & 0 \\
1 & 1 & 1 & 0 & 1 \\
\end{ytableau}$};
\node  at (0,-2) {TSSCPP Boolean triangle};

\node[below] at (-1,-.175) {\scriptsize $P_5$};
\node[below] at (-.55, .275) {\scriptsize $P_4$};
\node[below] at (-0.1, .7) {\scriptsize $P_3$};
\node[below] at (.35,1.15) {\scriptsize $P_2$};
\node[below] at (.8,1.6) {\scriptsize $P_1$};

\end{scope}

\end{tikzpicture}
    \caption{The fundamental domain of a TSSCPP of order 6 and its corresponding lattice path and  triangle representations.}
    
    \label{fig:lattice-path}
\end{figure}
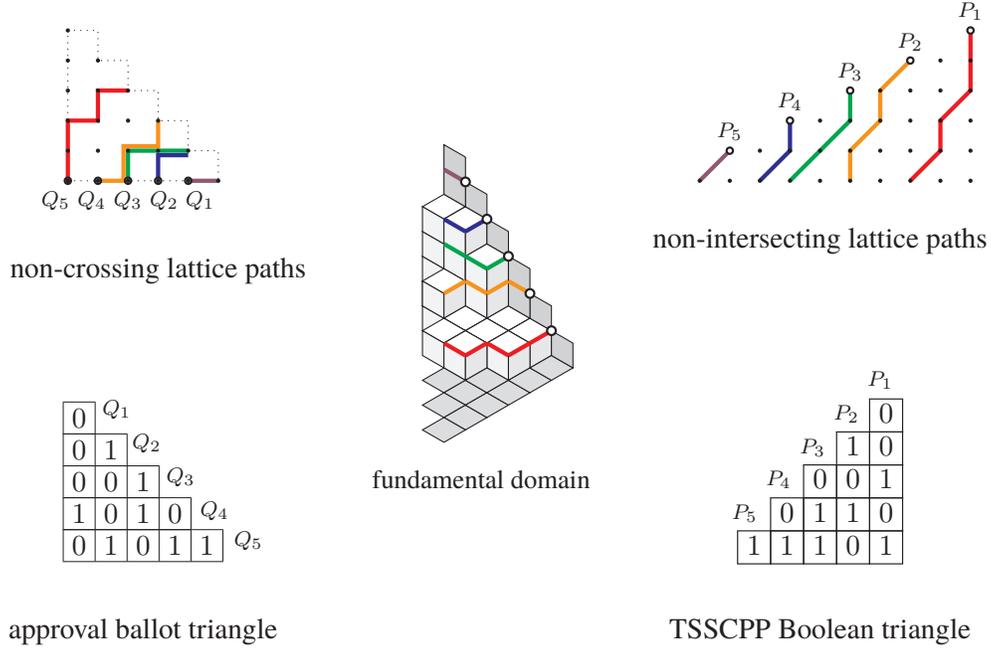

We obtain our second lattice path family using the top-down view of the fundamental domain (along with a rotation). From this vantage point, the boundary paths for the layers can touch at vertices and along edges, but they cannot cross. This family is encoded in the rows of an ABT of size 5, as shown by the paths $Q_1, Q_2, Q_3, Q_4, Q_5$ in the left hand side of Figure \ref{fig:lattice-path}.

\begin{definition}
\label{def:nclp}
Consider a sequence of paths $Q_1, \ldots, Q_{n}$ where path
$Q_i$ starts at $(n+1-i,1)$ and ends at the diagonal $D=\{ (n+2-j,j) : 1 \leq j \leq n+1 \}$, taking only  east $(1,0)$ steps and north $(0,1)$ steps. 
This sequence is a \emph{nest of non-crossing lattice paths} (NCLP) of size $n$ when the paths are pairwise non-crossing. In other words, for
$1 \leq i < j \leq n$, each lattice point on path $Q_i$ is  weakly southeast of each lattice point on path $Q_j$.
\end{definition}

\begin{proposition}
\label{prop:abt-noncrossing}
ABTs of order $n$ are in bijection with NCLPs of order $n$. 
\end{proposition}

\begin{proof}
Let $A \in \abt{n}$. Row $i$ of  $A$ encodes the path $Q_i$ \emph{in reverse}, so we read the row from right to left. A one entry corresponds to a north step and a zero entry corresponds to an east step. We take a total of $i$ steps from $(n+1-i,1)$, which ends at a point in the diagonal set $D$. 

We now show that adjacent paths do not cross. Consider paths $Q_{i}$ and $Q_{i+1}.$ Suppose that path $Q_{i+1}$ takes $s$ east steps during its  first $r+1$ steps, so that
$\sum_{k=i - r}^{i+1}  A(i+1,k) = r+1 - s$. By condition \eqref{eqn:abt-row},  we have $ \sum_{k=i-r}^{i} A(i,k) \leq r+1-s$. Therefore path $Q_{i}$ must take at least $s-1$ east steps in its first $r$ steps. 
Meanwhile, the starting point $(n+1-i,1)$ of path $Q_{i}$ is one step to the east of the starting point $(n-i,1)$ of $Q_{i+1}$. As a consequence, these two paths can overlap, but they cannot cross. 
\end{proof}

It is clear from Figure \ref{fig:lattice-path} that these lattice paths and triangles are in bijection: the non-intersecting path $P_i$ corresponds to non-crossing path $Q_{n-i}$. The paths are encoded by the \emph{columns} of a TSSCPP Boolean triangle and by the \emph{rows} of an ABT. We also note that these binary triangles use \emph{opposite} encodings of zeros and ones. 

It may be tempting to claim that the bijection between TSSCPPs and NCLPs is ``obvious.'' But Doran provides a rigorous proof of the bijection between  TSSCPPs and NILPs. We follow suit, and give an explicit mapping between TSSCPP Boolean triangles and ABTs. 

\begin{prop}
\label{prop:abt-tsscpp-boolean}
ABTs of size $n-1$ are in bijection with TSSCPP Boolean triangles of order $n$.
\end{prop}

\begin{proof}
Let $B$ be a TSSCPP Boolean triangle of order $n$ (which has $n-1$ rows) and let $A$ be the triangular array obtained by swapping 0's and 1's and then rotating clockwise by $\pi/2$. More formally,
\begin{align*}
A(i,j) &= 1 - B(n-j,i)  \quad \mbox{for}
\quad 1 \leq j \leq i \leq n-1, \\
B(j,k) &= 1 - A(k,n-j)  \quad \mbox{for}
\quad 1 \leq j \leq n-k \leq n-1.
\end{align*}
The column compatibility  equation \eqref{eqn:tsscpp-boolean} becomes the row compatibility  equation \eqref{eqn:abt-row}:
\begin{align*}
1+\sum_{r=t+1}^{s}  B(r,n-t-1) & \geq \sum_{r=t}^{s} B(r,n-t) \\
1+\sum_{r=t+1}^{s}  (1 - A(n-t-1,n-r)) & \geq \sum_{r=t}^{s} (1- A(n-t,n-r)) \\
\sum_{r=t+1}^{s}  A(n-t-1,n-r) & \leq \sum_{r=t}^{s} A(n-t,n-r) \\
\sum_{k=j}^{i}  A(i,k) & \leq \sum_{k=j}^{i+1} A(i+1,k) 
\end{align*}
where we successively change variables  $i=n-t-1$, then $k=n-r$ and finally  $j=n-s$.
Therefore, the image of this mapping is an ABT of size $n-1$. Reversing the argument confirms that the mapping is a bijection.
\end{proof}

We can now surmise that ABTs  are in bijection with TSSCPPs.

\begin{proof}[Proof of Theorem \ref{thm:abt-tsscpp}]
Follows directly from Propositions \ref{prop:tsscpp-boolean} and \ref{prop:abt-tsscpp-boolean}. 
\end{proof}


\section{Approval Ballot Hypertriangles}

\label{sec:ssb}

We prove Theorem \ref{thm:hypertriangle}:  a strict-sense ballot (SSB) can be decomposed into an approval ballot hypertriangle (ABH), which is a list  of compatible approval ballot triangles.

Recall that a strict-sense ballot has $n$ candidates where  candidate $k$ receives $n+1-k$ votes. We must order these votes so that  candidate $k$ always leads candidate $k+1$ for $1 \leq k \leq  n- 1$ during the vote count. 
SSBs of order $n$  are in bijection with shifted standard Young tableaux (shifted SYT) of staircase shape $(n,n-1,\ldots,1)$,
as shown in Figure \ref{fig:ssb}. We must use shifted SYT, rather than regular SYT,  to keep candidate $k$  strictly ahead of candidate $k+1$.
In a shifted staircase SYT, the first box in row $i > 1$ is located below the second box of row $i-1$. The integers $1,2, \ldots ,n(n+1)/2$ are arranged in the boxes so that the  rows and the columns are both increasing. The bijection from strict-sense ballots to  shifted SYT is straight-forward: row $i$ of the SYT contains the indices $S_i = \{ k :  s_k = i \}$ of the votes for candidate $i$, listed in increasing order. 

In order to match the orientation of ABTs, let us reverse both the rows and columns of our shifted SYT.  The result is a triangular array $T$ of shape $(1,2,\ldots, n)$ where row $i$  lists the vote indices $S_{n+1-i}$ in decreasing order; of course, this  triangle also has decreasing columns.
 For convenience, we will refer to $T$ as a \emph{SSB triangle} to distinguish this layout from a shifted SYT. Figure \ref{fig:ssb} shows  a strict-sense ballot for $7$ candidates and its corresponding shifted SYT and  SSB triangle.

\ytableausetup{boxsize=1.15em}

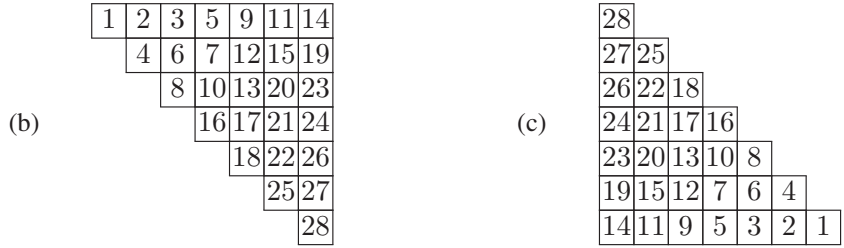
\begin{figure}[ht]
\begin{center}
\begin{tikzpicture}
\node at (0,3) {$(
1,1,1,2,1,
2,2,3,1,3,
1,2,3,1,2,
4,4,5,2,3,
4,5,3,4,6,
5,6,7)$};

\node at (-3.5,0) {
\begin{ytableau}
1 & 2 & 3 & 5 & 9 & 11 & 14 \\
\none & 4 & 6 & 7 & 12 & 15 & 19 \\
\none & \none & 8 & 10 & 13 & 20 & 23 \\
\none & \none & \none & 16 & 17 & 21 & 24 \\
\none & \none & \none & \none & 18 & 22 & 26 \\
\none & \none & \none & \none & \none & 25 & 27 \\
\none & \none & \none & \none & \none & \none & 28
\end{ytableau}
};

\node at (3.25,0) {
\begin{ytableau}
28 \\
27 & 25 \\
26 & 22 & 18 \\
24 & 21 & 17 & 16 \\
23 & 20 & 13 & 10 & 8 \\
19 & 15 & 12 & 7 & 6 & 4 \\
14 & 11  & 9  & 5 & 3 & 2  & 1
\end{ytableau}
};

\node at (-6,3) {\small (a)};
\node at (-6,0) {\small (b)};
\node at (0.75,0) {\small (c)};

\end{tikzpicture}
\end{center}

\caption{(a) A strict-sense ballot for 7 candidates. (b) Its corresponding shifted standard Young tableau of staircase shape. (c) Its corresponding SSB triangle  with decreasing rows and columns. }
\label{fig:ssb}
\end{figure}

We will prove that an SSB triangle of order $n$ corresponds to an approval ballot hypertriangle of size $n-2$ by creating a bijective mapping that turns an SSB triangle into a collection of non-crossing lattice paths that start from a triangular lattice of points. This requires extending Definition \ref{def:nclp} of \emph{non-crossing} to pairs of lattice paths that do not start in the same row, and we do so in Definition \ref{def:nclp-tri} below.

For convenience, we repeat Definition \ref{def:abt-hyper}. An   approval ballot hypertriangle of size $n$ is a sequence of
compatible ABTs $(A_{n}, A_{n-1}, \ldots, A_{1})$ where $A_{\ell} \in \mathcal{A}_{\ell}$. This triangle compatibility depends upon the entries $A_s(t,u)=1$ where $1 \leq u \leq t \leq s$. Each such entry (which corresponds to a north step of a NCLP) forces the triangle compatibility condition
\begin{equation}
\label{eqn:abt-hyper}
A_{s'} (t'-1,[1:t'-1]) \prec A_s (t,[1:t']) \preceq A_{s'}(t',[1:t'])
\end{equation}
where
$$
s' = s - \sum_{j=u}^t A(t,j) \quad \mbox{and} \quad t'=u-1,
$$
with the first inequality holding when $1 < t' < s'$ and the second holding when $1 \leq t' < s'$.

We will prove that the set of ABHs $(A_{n-2}, A_{n-3}, \ldots, A_{1})$ are in bijection with the SSB triangles  for $n$ candidates. 
Figure \ref{fig:abt-hyper} shows the ABH corresponding to the strict-sense ballot in Figure \ref{fig:ssb}. This figure also shows the collection of NCLPs that provide the mapping from the SSB triangle to the list of approval ballot triangles.

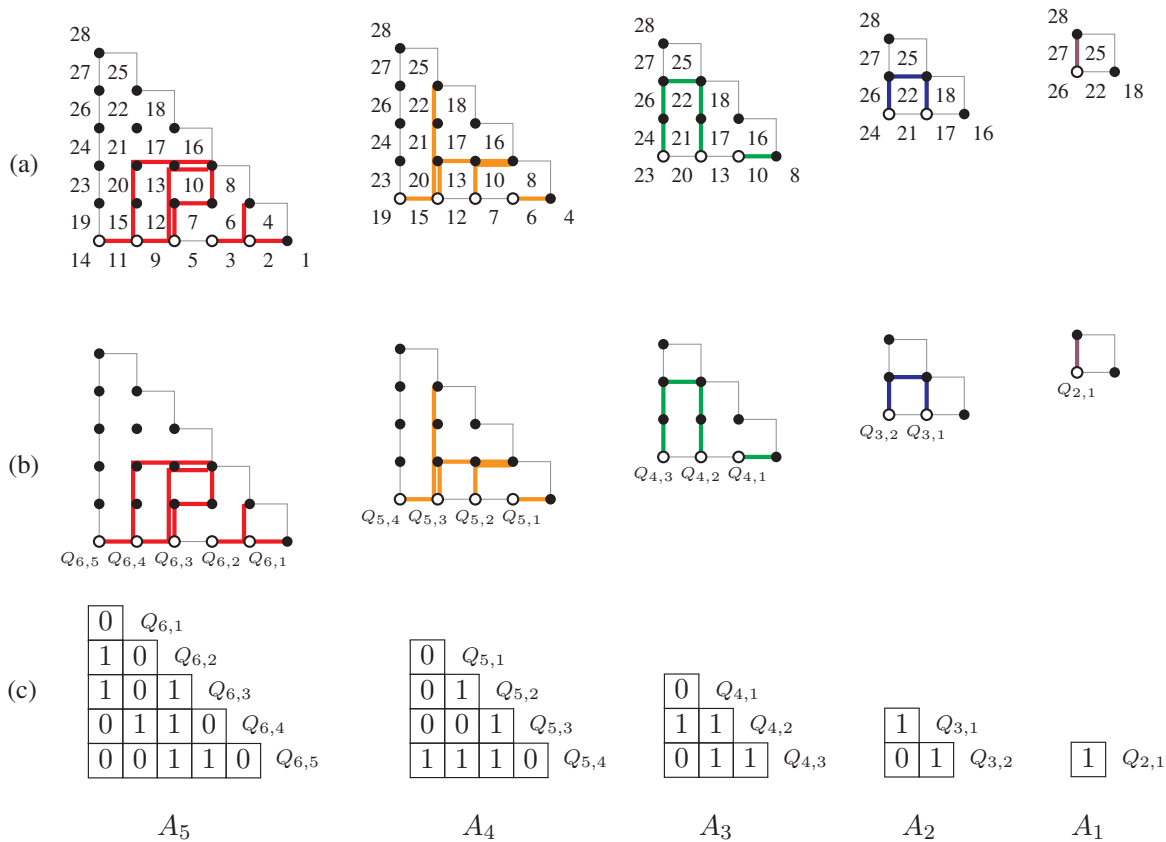
\begin{figure}[ht]

\begin{center}
\begin{tikzpicture}[scale=.5]

\node at (-2,2) {\small (a)};

\begin{scope}[shift={(0,0)}]

\draw[gray] (0,0) -- (5,0) -- (5,1) -- (4,1) -- (4,2) -- (3,2) -- (3,3) -- (2,3) -- (2,4) -- (1,4) -- (1,5) -- (0,5) -- cycle;

\draw[red, ultra thick] (0.0,0) -- (.9,0) -- (.9,2.1)  -- (3,2.1);
\draw[red, ultra thick] (1,0.0) -- (1.85,0.0) -- (1.85,1.9) -- (2.9,1.9);
\draw[red, ultra thick] (2.0,0) -- (2.0,1) -- (3,1) -- (3,2);
\draw[red, ultra thick](3,0) -- (3.85,0) -- (3.85,1);
\draw[red, ultra thick] (4,0.0) -- (5,0.0);

\fill (0, 0) circle (4pt);

\foreach \x in {0,1, ..., 5} {
\foreach \y in {0, ..., \x} {
    \fill (5- \x,  \y) circle (4pt);
}
}

\draw[thick, fill=white] (0,0) circle  (4pt);
\draw[thick, fill=white] (1.0,0) circle  (4pt);
\draw[thick, fill=white] (2,0) circle  (4pt);
\draw[thick, fill=white]  (3,0) circle  (4pt);
\draw[thick, fill=white]  (4.0,0) circle  (4pt);

\node at (-.5, -.5) {\scriptsize 14};
\node at (.5, -.5) {\scriptsize 11};
\node at (1.5, -.5) {\scriptsize 9};
\node at (2.5, -.5) {\scriptsize 5};
\node at (3.5, -.5) {\scriptsize 3};
\node at (4.5, -.5) {\scriptsize 2};
\node at (5.5, -.5) {\scriptsize 1};

\node at (-.5, .5) {\scriptsize 19};
\node at (.5, .5) {\scriptsize 15};
\node at (1.5, .5) {\scriptsize 12};
\node at (2.5, .5) {\scriptsize 7};
\node at (3.5, .5) {\scriptsize 6};
\node at (4.5, .5) {\scriptsize 4};

\node at (-.5, 1.5) {\scriptsize 23};
\node at (.5, 1.5) {\scriptsize 20};
\node at (1.5, 1.5) {\scriptsize 13};
\node at (2.5, 1.5) {\scriptsize 10};
\node at (3.5, 1.5) {\scriptsize 8};

\node at (-.5, 2.5) {\scriptsize 24};
\node at (.5, 2.5) {\scriptsize 21};
\node at (1.5, 2.5) {\scriptsize 17};
\node at (2.5, 2.5) {\scriptsize 16};

\node at (-.5, 3.5) {\scriptsize 26};
\node at (.5, 3.5) {\scriptsize 22};
\node at (1.5, 3.5) {\scriptsize 18};

\node at (-.5, 4.5) {\scriptsize 27};
\node at (.5, 4.5) {\scriptsize 25};

\node at (-.5, 5.5) {\scriptsize 28};
\end{scope}

\begin{scope}[shift={(8,.125)}]

\draw[gray] (0,1) -- (4,1) -- (4,2) -- (3,2) -- (3,3) -- (2,3) -- (2,4) -- (1,4) -- (1,5) -- (0,5) -- cycle;

\draw[orange, ultra thick] (0.0,1) -- (.9,1) -- (.9,4);
\draw[orange, ultra thick] (1.05,1) --  (1.05,2) -- (3,2);
\draw[orange, ultra thick] (2.0,1) -- (2,1.9) -- (2.95,1.9);
\draw[orange, ultra thick] (3,1) -- (4,1);

\foreach \x in {1, ..., 5} {
\foreach \y in {1, ..., \x} {
    \fill (5- \x,  \y) circle (4pt);
}
}

\draw[thick, fill=white] (0,1) circle  (4pt);
\draw[thick, fill=white] (1,1) circle  (4pt);
\draw[thick, fill=white] (2,1) circle  (4pt);
\draw[thick, fill=white]  (3,1) circle  (4pt);

\node at (-.5, .5) {\scriptsize 19};
\node at (.5, .5) {\scriptsize 15};
\node at (1.5, .5) {\scriptsize 12};
\node at (2.5, .5) {\scriptsize 7};
\node at (3.5, .5) {\scriptsize 6};
\node at (4.5, .5) {\scriptsize 4};

\node at (-.5, 1.5) {\scriptsize 23};
\node at (.5, 1.5) {\scriptsize 20};
\node at (1.5, 1.5) {\scriptsize 13};
\node at (2.5, 1.5) {\scriptsize 10};
\node at (3.5, 1.5) {\scriptsize 8};

\node at (-.5, 2.5) {\scriptsize 24};
\node at (.5, 2.5) {\scriptsize 21};
\node at (1.5, 2.5) {\scriptsize 17};
\node at (2.5, 2.5) {\scriptsize 16};

\node at (-.5, 3.5) {\scriptsize 26};
\node at (.5, 3.5) {\scriptsize 22};
\node at (1.5, 3.5) {\scriptsize 18};

\node at (-.5, 4.5) {\scriptsize 27};
\node at (.5, 4.5) {\scriptsize 25};

\node at (-.5, 5.5) {\scriptsize 28};

\end{scope}

\begin{scope}[shift={(15,.25)}]

\draw[gray] (0,2) -- (3,2) -- (3,3) -- (2,3) -- (2,4) -- (1,4) -- (1,5) -- (0,5) -- cycle;
\draw[green, ultra thick] (0.0,2) -- (0,4)  -- (1,4);
\draw[green, ultra thick] (1,2) -- (1,4);
\draw[green, ultra thick] (2.0,2) -- (2.95,2);

\foreach \x in {2, ..., 5} {
\foreach \y in {2, ..., \x} {
    \fill (5- \x,  \y) circle (4pt);
}
}

\draw[thick, fill=white] (0,2) circle  (4pt);
\draw[thick, fill=white] (1,2) circle  (4pt);
\draw[thick, fill=white] (2,2) circle  (4pt);

\node at (-.5, 1.5) {\scriptsize 23};
\node at (.5, 1.5) {\scriptsize 20};
\node at (1.5, 1.5) {\scriptsize 13};
\node at (2.5, 1.5) {\scriptsize 10};
\node at (3.5, 1.5) {\scriptsize 8};

\node at (-.5, 2.5) {\scriptsize 24};
\node at (.5, 2.5) {\scriptsize 21};
\node at (1.5, 2.5) {\scriptsize 17};
\node at (2.5, 2.5) {\scriptsize 16};

\node at (-.5, 3.5) {\scriptsize 26};
\node at (.5, 3.5) {\scriptsize 22};
\node at (1.5, 3.5) {\scriptsize 18};

\node at (-.5, 4.5) {\scriptsize 27};
\node at (.5, 4.5) {\scriptsize 25};

\node at (-.5, 5.5) {\scriptsize 28};

\end{scope}

\begin{scope}[shift={(21,.375)}]

\draw[gray] (0,3) -- (2,3) -- (2,4) -- (1,4) -- (1,5) -- (0,5) -- cycle;
\draw[blue, ultra thick] (0.0,3) -- (0,4)  -- (1,4);
\draw[blue, ultra thick] (1,3) -- (1,4);

\foreach \x in {3, ..., 5} {
\foreach \y in {3, ..., \x} {
    \fill (5- \x,  \y) circle (4pt);
}
}

\draw[thick, fill=white] (0,3) circle  (4pt);
\draw[thick, fill=white] (1,3) circle  (4pt);

\node at (-.5, 2.5) {\scriptsize 24};
\node at (.5, 2.5) {\scriptsize 21};
\node at (1.5, 2.5) {\scriptsize 17};
\node at (2.5, 2.5) {\scriptsize 16};

\node at (-.5, 3.5) {\scriptsize 26};
\node at (.5, 3.5) {\scriptsize 22};
\node at (1.5, 3.5) {\scriptsize 18};

\node at (-.5, 4.5) {\scriptsize 27};
\node at (.5, 4.5) {\scriptsize 25};

\node at (-.5, 5.5) {\scriptsize 28};

\end{scope}

\begin{scope}[shift={(26,.5)}]

\draw[gray] (0,4) -- (1,4) -- (1,5) -- (0,5) -- cycle;
\draw[violet, ultra thick] (0,4)  -- (0,5);

\foreach \x in {4, ..., 5} {
\foreach \y in {4, ..., \x} {
    \fill (5- \x,  \y) circle (4pt);
}
}

\draw[thick, fill=white] (0,4) circle  (4pt);

\node at (-.5, 3.5) {\scriptsize 26};
\node at (.5, 3.5) {\scriptsize 22};
\node at (1.5, 3.5) {\scriptsize 18};

\node at (-.5, 4.5) {\scriptsize 27};
\node at (.5, 4.5) {\scriptsize 25};

\node at (-.5, 5.5) {\scriptsize 28};

\end{scope}

\begin{scope}[shift={(0,-8)}]

\node at (-2,2) {\small (b)};

\begin{scope}[shift={(0,0)}]

\draw[gray] (0,0) -- (5,0) -- (5,1) -- (4,1) -- (4,2) -- (3,2) -- (3,3) -- (2,3) -- (2,4) -- (1,4) -- (1,5) -- (0,5) -- cycle;

\draw[red, ultra thick] (0.0,0) -- (.9,0) -- (.9,2.1)  -- (3,2.1);
\draw[red, ultra thick] (1,0.0) -- (1.85,0.0) -- (1.85,1.9) -- (2.9,1.9);
\draw[red, ultra thick] (2.0,0) -- (2.0,1) -- (3,1) -- (3,2);
\draw[red, ultra thick](3,0) -- (3.85,0) -- (3.85,1);
\draw[red, ultra thick] (4,0.0) -- (5,0.0);

\fill (0, 0) circle (2pt);

\foreach \x in {0,1, ..., 5} {
\foreach \y in {0, ..., \x} {
    \fill (5- \x,  \y) circle (4pt);
}
}

\draw[thick, fill=white] (0,0) circle  (4pt);
\draw[thick, fill=white] (1.0,0) circle  (4pt);
\draw[thick, fill=white] (2,0) circle  (4pt);
\draw[thick, fill=white]  (3,0) circle  (4pt);
\draw[thick, fill=white]  (4.0,0) circle  (4pt);

\node at (-.5, -.5) {\tiny $Q_{6,5}$};
\node at (.75, -.5) {\tiny $Q_{6,4}$};
\node at (2, -.5) {\tiny $Q_{6,3}$};
\node at (3.25, -.5) {\tiny $Q_{6,2}$};
\node at (4.5, -.5) {\tiny $Q_{6,1}$};

\end{scope}

\begin{scope}[shift={(8,.125)}]

\draw[gray] (0,1) -- (4,1) -- (4,2) -- (3,2) -- (3,3) -- (2,3) -- (2,4) -- (1,4) -- (1,5) -- (0,5) -- cycle;

\draw[orange, ultra thick] (0.0,1) -- (.9,1) -- (.9,4);
\draw[orange, ultra thick] (1.05,1) --  (1.05,2) -- (3,2);
\draw[orange, ultra thick] (2.0,1) -- (2,1.9) -- (2.95,1.9);
\draw[orange, ultra thick] (3,1) -- (4,1);

\foreach \x in {1, ..., 5} {
\foreach \y in {1, ..., \x} {
    \fill (5- \x,  \y) circle (4pt);
}
}

\draw[thick, fill=white] (0,1) circle  (4pt);
\draw[thick, fill=white] (1,1) circle  (4pt);
\draw[thick, fill=white] (2,1) circle  (4pt);
\draw[thick, fill=white]  (3,1) circle  (4pt);

\node at (-.5, .5) {\tiny $Q_{5,4}$};
\node at (.75, .5) {\tiny $Q_{5,3}$};
\node at (2, .5) {\tiny $Q_{5,2}$};
\node at (3.25, .5) {\tiny $Q_{5,1}$};

\end{scope}

\begin{scope}[shift={(15,.25)}]

\draw[gray] (0,2) -- (3,2) -- (3,3) -- (2,3) -- (2,4) -- (1,4) -- (1,5) -- (0,5) -- cycle;


\draw[green, ultra thick] (0.0,2) -- (0,4)  -- (1,4);
\draw[green, ultra thick] (1,2) -- (1,4);
\draw[green, ultra thick] (2.0,2) -- (2.95,2);

\foreach \x in {2, ..., 5} {
\foreach \y in {2, ..., \x} {
    \fill (5- \x,  \y) circle (4pt);
}
}

\draw[thick, fill=white] (0,2) circle  (4pt);
\draw[thick, fill=white] (1,2) circle  (4pt);
\draw[thick, fill=white] (2,2) circle  (4pt);

\node at (-.25, 1.5) {\tiny $Q_{4,3}$};
\node at (1, 1.5) {\tiny $Q_{4,2}$};
\node at (2.25, 1.5) {\tiny $Q_{4,1}$};

\end{scope}

\begin{scope}[shift={(21,.375)}]

\draw[gray] (0,3) -- (2,3) -- (2,4) -- (1,4) -- (1,5) -- (0,5) -- cycle;
\draw[blue, ultra thick] (0.0,3) -- (0,4)  -- (1,4);
\draw[blue, ultra thick] (1,3) -- (1,4);

\foreach \x in {3, ..., 5} {
\foreach \y in {3, ..., \x} {
    \fill (5- \x,  \y) circle (4pt);
}
}

\draw[thick, fill=white] (0,3) circle  (4pt);
\draw[thick, fill=white] (1,3) circle  (4pt);

\node at (-.25, 2.5) {\tiny $Q_{3,2}$};
\node at (1, 2.5) {\tiny $Q_{3,1}$};

\end{scope}

\begin{scope}[shift={(26,.5)}]

\draw[gray] (0,4) -- (1,4) -- (1,5) -- (0,5) -- cycle;
\draw[violet, ultra thick] (0,4)  -- (0,5);

\foreach \x in {4, ..., 5} {
\foreach \y in {4, ..., \x} {
    \fill (5- \x,  \y) circle (4pt);
}
}

\draw[thick, fill=white] (0,4) circle  (4pt);

\node at (0, 3.5) {\tiny $Q_{2,1}$};

\end{scope}

\end{scope}


\begin{scope}[shift={(2,-12)}, scale=.9]

\node at  (-4.5,0) {\small (c)};

\node at (0,0) {
\begin{ytableau}
0 \\
1 & 0 \\
1 & 0 & 1 \\
0 & 1 & 1 & 0  \\
0 & 0 & 1 & 1 & 0
\end{ytableau}
};

\node at (0, -4) {$A_5$};

\node at (-.35,2) {\scriptsize $Q_{6,1}$};
\node at (.65,1) {\scriptsize $Q_{6,2}$};
\node at (1.65,0) {\scriptsize $Q_{6,3}$};
\node at (2.65,-1) {\scriptsize $Q_{6,4}$};
\node at (3.65,-2) {\scriptsize $Q_{6,5}$};

\node at (9,-.5) {
\begin{ytableau}
0 \\ 
0 & 1 \\
0 & 0 & 1 \\
1 & 1 & 1 & 0 
\end{ytableau}
};

\node at (9, -4) {$A_4$};

\begin{scope}[shift={(8.5,0)}]
\node at (.65,1) {\scriptsize $Q_{5,1}$};
\node at (1.65,0) {\scriptsize $Q_{5,2}$};
\node at (2.65,-1) {\scriptsize $Q_{5,3}$};
\node at (3.65,-2) {\scriptsize $Q_{5,4}$};
\end{scope}

\node at (16,-1) {
\begin{ytableau}
0 \\ 
1 & 1 \\
0 & 1 & 1 \\
\end{ytableau}
};

\node at (16, -4) {$A_3$};

\begin{scope}[shift={(15,0)}]
\node at (1.65,0) {\scriptsize $Q_{4,1}$};
\node at (2.65,-1) {\scriptsize $Q_{4,2}$};
\node at (3.65,-2) {\scriptsize $Q_{4,3}$};
\end{scope}

\node at (22,-1.5) {
\begin{ytableau}
1 \\ 
0 & 1 \\
\end{ytableau}
};

\node at (22, -4) {$A_2$};

\begin{scope}[shift={(20.5,0)}]
\node at (2.65,-1) {\scriptsize $Q_{3,1}$};
\node at (3.65,-2) {\scriptsize $Q_{3,2}$};
\end{scope}

\node at (27,-2) {
\begin{ytableau}
1 \\ 
\end{ytableau}
};

\node at (27, -4) {$A_1$};

\begin{scope}[shift={(25,0)}]
\node at (3.65,-2) {\scriptsize $Q_{2,1}$};
\end{scope}

\end{scope}

\end{tikzpicture}
\end{center}

\caption{Mapping a SSB triangle of order $n=7$ to an ABH. (a) The NCLPs $\mathcal{Q}$ for the SSB from  Figure \ref{fig:ssb}. These NCLP are pairwise noncrossing as well. (b) The NCLP $\mathcal{Q} = \{ Q_{i,j}  : 1 \leq j <  i \leq 6 \}$ with their labels. (c) The ABT $A_{r-1}$ corresponds to subset $\mathcal{Q}_r = \{ Q_{r,j} : 1 \leq j < r \}$ for
$2 \leq  r \leq 6$.}
\label{fig:abt-hyper}
\end{figure}

\begin{proof}[Proof of Theorem \ref{thm:hypertriangle}]
Consider a strict-sense ballot for $n$ candidates with  SSB triangle $T$. For each $3 \leq r \leq n$, we will construct an ABT $A_{r-2}$ of size $r-2$ that describes how row $r$ of $T$ compares to the previous $r-1$ rows of $T$. 
We start by drawing a collection $\mathcal{Q} = \{ Q_{i,j} : 1 \leq j < i \leq n-1 \}$ of noncrossing lattice paths,  as shown in Figure \ref{fig:abt-hyper}. 
We construct these paths as follows.
 
\begin{enumerate}
\item Associate entry $T(r,j)$ with the northeast corner of its boundary square for $1 \leq j < r \leq n$. 
\item Label the corner of $T(r,j)$ by $(r-1, r-1-j)$. These labeled corners will be our lattice points.
\item Starting at corner $(r-1, r-1-j)$, draw the unique lattice path $Q_{r-1, r-1-j}$ to the main diagonal so that entries $T(s,k)$ larger (resp.~smaller) than entry $T(r,j)$ are northwest (resp.~southeast) of this lattice path.
\end{enumerate}

\begin{remark}
Note that the labeling of our lattice points is unconventional, as shown in Figure \ref{fig:labels-corners}. We have organized the points by rows and diagonals, rather than rows and columns. Furthermore, our label entries \emph{decrease} when we travel north or east. The important feature to note is that a north step is now $(-1,-1)$ and an east step is now $(0,-1)$. 
\end{remark}

\ytableausetup{nosmalltableaux}

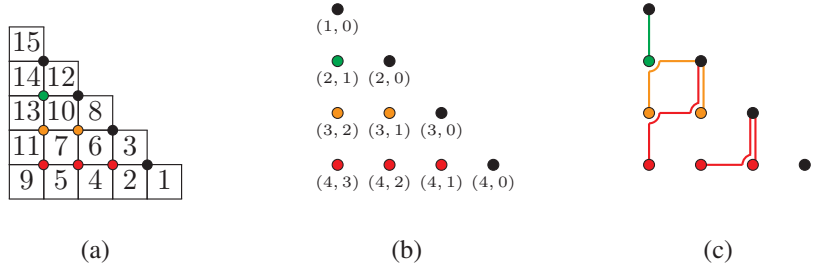
\begin{figure}
 
 \centering
 \begin{tikzpicture}[scale=.46]
 

 \node at (0,0) {
\begin{ytableau}
15 \\
14 & 12 \\
13 & 10 & 8 \\
11 & 7 & 6 & 3  \\
9 & 5 & 4 & 2 & 1
\end{ytableau}
};

\node at (0,-4) {\small (a)}; 
\node at (9,-4) {\small (b)}; 
\node at (18,-4) {\small (c)};  
 
\begin{scope}[shift={(-1.5,-2.5)}]
 
 \foreach \x in {2, ..., 5} {
\foreach \y in {2, ..., \x} {
    \fill (5- \x,  \y-1) circle (3pt);
}
}

\draw[fill] (0,4) circle (4pt);

\draw[fill=green] (0,3) circle (4pt);
\draw[fill] (1,3) circle (4pt);

\draw[fill=orange] (0,2) circle (4pt);
\draw[fill=orange] (1,2) circle (4pt);
\draw[fill] (2,2) circle (4pt);

\draw[fill=red] (0,1) circle (4pt);
\draw[fill=red] (1,1) circle (4pt);
\draw[fill=red] (2,1) circle (4pt);
\draw[fill] (3,1) circle (4pt);
\end{scope}

\begin{scope}[shift={(7,-3)}, scale = 1.5]

     \draw[fill] (0,4) circle (3pt);
     \node[below] at (0,4) {\tiny $(1,0)$};

     \draw[fill=green] (0,3) circle (3pt);
     \node[below] at (0,3) {\tiny $(2,1)$};     
     \draw[fill] (1,3) circle (3pt);
     \node[below] at (1,3) {\tiny $(2,0)$};

     \draw[fill=orange] (0,2) circle (3pt);
     \node[below] at (0,2) {\tiny $(3,2)$};     
     \draw[fill=orange] (1,2) circle (3pt);
     \node[below] at (1,2) {\tiny $(3,1)$};
    \draw[fill]  (2,2) circle (3pt);
     \node[below] at (2,2) {\tiny $(3,0)$};   
     
     \draw[fill=red] (0,1) circle (3pt);
     \node[below] at (0,1) {\tiny $(4,3)$};     
     \draw[fill=red] (1,1) circle (3pt);
     \node[below] at (1,1) {\tiny $(4,2)$};
     \draw[fill=red] (2,1) circle (3pt);
     \node[below] at (2,1) {\tiny $(4,1)$};        
     \draw[fill]  (3,1) circle (3pt);
     \node[below] at (3,1) {\tiny $(4,0)$};

\end{scope}

\begin{scope}[shift={(16,-3)}, scale = 1.5]

\draw[thick,red] (0,1) -- (0,1.8);
\draw[thick,red] (0.2,2) -- (0.8,2);
\draw[thick,red] (.95,2.2) -- (.95,3);

\draw[thick,red] (1,1) -- (1.8,1); 
\draw[thick,red] (1.95,1.2) -- (1.95,2); 
\draw[thick,red] (2.05,1) -- (2.05,2);  

\draw[thick,orange] (0,2) -- (0,2.8); 
\draw[thick,orange] (.2,3) -- (1,3); 
\draw[thick,orange] (1.05,2) -- (1.05,3);  

\draw[thick,green] (0,3) -- (0,4);  
 
     \draw[fill] (0,4) circle (3pt);

     \draw[fill=green] (0,3) circle (3pt);
     \draw[fill] (1,3) circle (3pt);

     \draw[fill=orange] (0,2) circle (3pt);
     \draw[fill=orange] (1,2) circle (3pt);
     \draw[fill] (2,2) circle (3pt);
     
     \draw[fill=red] (0,1) circle (3pt);
     \draw[fill=red] (1,1) circle (3pt);
     \draw[fill=red] (2,1) circle (3pt);
     \draw[fill] (3,1) circle (3pt);

 \begin{scope}[shift={(0,3)}]
  \draw[thick,orange] (0,-.2) arc[start angle=-90, end angle=0,radius=.2];
 \end{scope} 

 \begin{scope}[shift={(0,2)}]
  \draw[thick,red] (0,-.2) arc[start angle=-90, end angle=0,radius=.2];
 \end{scope}

 \begin{scope}[shift={(1,2)}]
  \draw[thick,red] (-.2,0) arc[start angle=180, end angle=100,radius=.2];
 \end{scope} 

  \begin{scope}[shift={(2,1)}]
  \draw[thick,red] (-.2,0) arc[start angle=180, end angle=100,radius=.2];
 \end{scope}
 
\end{scope} 
 
 \end{tikzpicture}
 
 \caption{(a) A SSB triangle of order 5. (b) Label the corner of SSB triangle entry $T(r,j)$ by $(r-1, r-1-j)$ for $1 \leq j < r \leq n$. A north step updates the label by $(-1,-1)$ and an east step updates the label by $(0,-1)$. (c) The NCLP of size 3 for this SSB triangle.}
 
 \label{fig:labels-corners}
 
 \end{figure}

The decreasing rows and columns of SSB triangle $T$ guarantee that the paths $\mathcal{Q}$ are pairwise noncrossing, as per the following definition.
 
 \begin{definition}
 \label{def:nclp-tri}
Let $(s,t)$ and $(u,v)$ be the labels for points in our triangular lattice, where $u \geq s$.
Let $Q_{s,t}$ and $P_{u,v}$ be lattice paths starting at these points. If path $P_{u,v}$ does not reach row $s$, then the paths are noncrossing. Otherwise, let $(s,w)$ be the label of the first point of $P_{u,v}$ in row $s$, and let $R_{s,w}$ be the subpath of $P_{u,v}$ that starts at this point. Then the paths $Q_{s,t}$ and $P_{u,v}$ are noncrossing when either (1) $w > t$ and $R_{s,w}$ is weakly northwest of $Q_{s,t}$, or (2) $w \leq t$ and $R_{s,w}$ is weakly southeast of $Q_{s,t}$. 
 \end{definition}

\begin{remark}
We find the following visualization helpful when confirming that  lattice paths $P,Q$ are noncrossing.
 When a lattice path $Q_{a,b}$ encounters a lattice point labeled $(c,d)$, rather than traveling through $(c,d)$, we \emph{circumvent} the point by traveling along a small circular arc. The direction of rotation depends upon the step that led to the point, see Figure \ref{fig;labels-corners}. If we stepped east to $(c,d)$, then we travel clockwise around $(c,d)$. Consequently, path $Q_{c,d}$  must remain southeast of $Q_{a,b}$. On the other hand, if we stepped north to $(c,d)$, then we travel counterclockwise  around $(c,d)$, and path $Q_{c,d}$ must remain northwest of $Q_{a,b}$.
 \end{remark}

 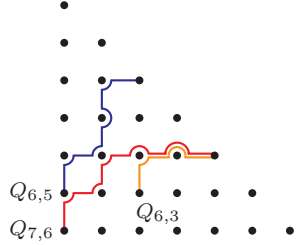
\begin{figure}[ht]
 
 \centering
 \begin{tikzpicture}[scale=.5]
 
   \begin{scope}[shift={(0,1)}, scale = .25]
  \draw[thick,red] (0,-1) arc[start angle=-90, end angle=0,radius=1];
 \end{scope}
 
 \begin{scope}[shift={(1,1)}, scale = .25]
  \draw[thick,red] (-1,0) arc[start angle=180, end angle=90,radius=1];
 \end{scope}
 
  \begin{scope}[shift={(1,2)}, scale = .25]
  \draw[thick,red] (0,-1) arc[start angle=-90, end angle=0,radius=1];
 \end{scope}
 
  \begin{scope}[shift={(2,2)}, scale = .25]
  \draw[thick,red] (-1,0) arc[start angle=180, end angle=10,radius=1];
 \end{scope}

    \begin{scope}[shift={(3,2.025)}, scale = .25]
  \draw[thick,red] (-1.27,0) arc[start angle=180, end angle=0,radius=1.27];
 \end{scope}

%
 
\draw[thick,red] (0,0) -- (0,0.775);
 \draw[thick,red] (0.225,1) -- (0.775,1);
 \draw[thick,red] (1,1.225) -- (1,1.775);
 \draw[thick,red] (1.225,2) -- (1.775,2);
 
   \draw[thick,red] (2.22,2.05) -- (2.69,2.05) -- (2.69, 2.1);
 \draw[thick,red] (3.30, 2.1) -- (3.31,2.05) -- (4,2.05);

 \begin{scope}[shift={(0,2)}, scale = .25]
  \draw[thick,blue] (0,-1) arc[start angle=-90, end angle=0,radius=1];
 \end{scope}

 \begin{scope}[shift={(1,2)}, scale = .25]
  \draw[thick,blue] (-1,0) arc[start angle=180, end angle=90,radius=1];
 \end{scope}
 
 \begin{scope}[shift={(1,3)}, scale = .25]
  \draw[thick,blue] (0,-1) arc[start angle=-90, end angle=90,radius=1];
 \end{scope} 
 
  \begin{scope}[shift={(1,4)}, scale = .25]
  \draw[thick,blue] (0,-1) arc[start angle=-90, end angle=0,radius=1];
 \end{scope}

 \draw[thick,blue] (0,1) -- (0,1.775);
 \draw[thick,blue] (0.225,2) -- (0.775,2);
 \draw[thick,blue] (1,2.225) -- (1,2.775);
  \draw[thick,blue] (1,3.225) -- (1,3.775);
 \draw[thick,blue] (1.225,4) -- (2,4);

\begin{scope}[shift={(2,2)}, scale = .25]
  \draw[thick,orange] (0,-1) arc[start angle=-90, end angle=-10,radius=1];
 \end{scope}

    \begin{scope}[shift={(3,2.0)}, scale = .25]
  \draw[thick,orange] (-0.9,0) arc[start angle=180, end angle=0,radius=.9];
 \end{scope}
 
  \draw[thick,orange] (2,1) -- (2,1.775);
  \draw[thick,orange] (2.22,1.95) -- (2.78,1.95) -- (2.78,2);
 \draw[thick,orange] (3.22,2) -- (3.22,1.95) -- (4,1.95);

 \foreach \x in {0,1,2,3,4,5,6}
 {
 \foreach \y in {0,...,\x}
 {
 \fill (6-\x, \y) circle (3pt);
 }
 }
 
 \node[left] at (0,0) {\scriptsize $Q_{7,6}$};
  \node[left] at (0,1) {\scriptsize $Q_{6,5}$};
 \node[below] at (2.5,1) {\scriptsize $Q_{6,3}$};  
 
 \end{tikzpicture}
 
 \caption{noncrossing lattice paths starting in different rows. Path $Q_{6,5}$ is weakly northwest of $Q_{7,6}$ and path $Q_{6,3}$ is weakly southeast of $Q_{7,6}$. The noncrossing condition can be visualized by drawing paths that circumnavigate intermediate lattice points. East steps move clockwise and north steps move counterclockwise around the intermediate points. }
 \label{fig;labels-corners}
 
 \end{figure}

We are now ready to create our approval ballot hypertriangle. For $3 \leq r \leq n$,
we use the  NCLP for row $r$ of $T$ to define ABT $A_{r-2} \in \mathcal{A}_{r-2}$ in the usual way (see the proof of Proposition \ref{prop:abt-noncrossing}). That is, row $i$ of $A_{r-2}$ encodes path $Q_{r-1,i}$ when this row is read \emph{right-to-left}: the ones and zeros encode north steps and east steps, respectively. Figure \ref{fig:abt-hyper}(c) shows an example of the resulting ABTs.

 We now have a list $(A_{n-2}, A_{n-3}, \ldots, A_{1})$ where $A_k \in \mathcal{A}_k$. Since \emph{every} pair of lattice paths $Q_{s,t}, Q_{s',t'}$ is noncrossing, 
 our next task is to encode this pairwise noncrossing property as an algebraic condition on our list of ABTs.  We will guarantee that $Q_{s,t}$ does not cross any other path by adding (at most) two noncrossing constraints for each north step of $Q_{s,t}$.

 Recall that path $Q_{s,t}$ originates at  corner  $(s,t)$ where $s > t \geq 1$, and that the entries on the labels of the corners \emph{decrease} as we traverse the path. In particular, a north step is $(-1,-1)$ and an east step is $(0,-1)$.
Suppose that path $Q_{s,t}$ steps northward from $(s'+1,t'+1)$ to $(s',t')$ as shown in Figure \ref{fig:north-step}(a). Consider the paths $\mathcal{Q}_{s'} = \{ Q_{s', s'-1}, \ldots, Q_{s',1} \}$. The remaining $t'$ steps of path $Q_{s,t}$ impact the paths $\mathcal{Q}_{s'}$. More precisely, (a) the paths $Q_{s',s'-1},  \ldots ,Q_{s',t'}$ must remain weakly northwest of the last $t'$ steps of $Q_{s,t}$; and (b) the paths $Q_{s',t'-1},  \ldots , Q_{s',1}$ must remain weakly southeast of the last $t'$ steps of $Q_{s,t}$. 
 The paths of $\mathcal{Q}_{s'}$ are pairwise noncrossing, and this is encoded in $A_{s'}$. Therefore, it is sufficient to ensure that  $Q_{s',t'}$ is northeast of $Q_{s,t}$ and that $Q_{s',t'-1}$ is southeast of $Q_{s,t}$, see Figure \ref{fig:north-step}(b).

\begin{figure}

\begin{center}

\begin{tikzpicture}[scale=0.5]

\tikzset{decoration=snake}

\begin{scope}

\draw (0,0) -- (8,0) -- (0,8) -- cycle;

\fill[fill=gray!45] (2,3) -- (5,3) -- (3.5,4.5)   decorate{--(2,3)} -- (5, 3) ;

\fill[fill=gray!15]  (0,3) -- (0,8) -- (3.5,4.5)  decorate{--(2,3)};

\node[right] at (4.5,4.25) {\scriptsize $(s',t'-1)$};
\node[left] at (-0.5,2.8) {\scriptsize $(s'+1,t'+1)$};

\node[left] at (-0.5,4.5) {\scriptsize $(s',t')$};

\draw[-latex] (-0.5, 2.8) -- (1.9,2.6);
\draw[-latex] (4.5, 4.15) -- (2.6,3.1);
\draw[-latex] (-0.5, 4.5) -- (1.8,3.1);

\node at (2.15,1.5) {\scriptsize $Q_{s,t}$};

\foreach \x in {0,.5,...,4.5} {
	\draw[fill] (\x,3) circle (2pt);
}

\draw[very thick] (0.5,1) decorate{--(2,2.5)} -- (2,3);
\draw[very thick] (3.5,4.5) decorate{--(2,3)};

\draw[fill] (0.5,1) circle (2pt);
\node[left] at (-0.5,0.5) {\scriptsize $(s,t)$};
\draw[-latex] (-0.5, 0.5) -- (0.4,0.9);

\draw[fill] (2,3) circle (2.5pt);
\draw[fill] (2,2.5) circle (2.5pt);

\node at (4, -1) {\small (a)};

\end{scope}

\begin{scope}[shift={(14,0)}]

\draw (0,0) -- (8,0) -- (0,8) -- cycle;

\fill[fill=gray!45] (2,3) -- (5,3) -- (3.5,4.5)   decorate{--(2,3)} -- (5, 3) ;

\fill[fill=gray!15]  (0,3) -- (0,8) -- (3.5,4.5)  decorate{--(2,3)};

\node at (2.15,1.5) {\scriptsize $Q_{s,t}$};

\foreach \x in {0,.5,...,4.5} {
	\draw[fill] (\x,3) circle (2pt);
}

\draw (0.5,1) decorate{--(2,2.5)} -- (2,3);
\draw (3.5,4.5) decorate{--(2,3)};
\draw[fill] (0.5,1) circle (2pt);

\draw[fill] (2,3) circle (2.5pt);
\draw[fill] (2,2.5) circle (2.5pt);

\draw[very thick] (4,4)   decorate{--(2.5,3)};

\draw[very thick]  (2,3)   decorate{--(2.5,5.5)};

\draw[-latex] (5.5, 3.5) -- (3.75,3.5);
\draw[-latex] (-1, 4.5) -- (1.75,4);

\node[left] at (-0.5,5) {\scriptsize $Q_{s',t'}$};

\node[right] at (5.5,3.5) {\scriptsize $Q_{s',t'-1}$};

\node at (4, -1) {\small (b)};

\end{scope}

\end{tikzpicture}

\end{center}

\caption{(a) The path $Q_{s,t}$ starts at corner $(s,t)$ and has a north step from corner $(s'+1,t'+1)$ to corner $(s',t')$. (b) The path $Q_{s',t'}$ must be weakly northwest of $Q_{s,t}$. The path $Q_{s',t'-1}$ must be weakly southeast of $Q_{s,t}$. The remaining paths $Q_{s',j}$ starting from corners in row $s'$ will not cross $Q_{s,t}$ since they do not cross either $Q_{s',t'}$ or $Q_{s',t'-1}$.}

\label{fig:north-step}
\end{figure}
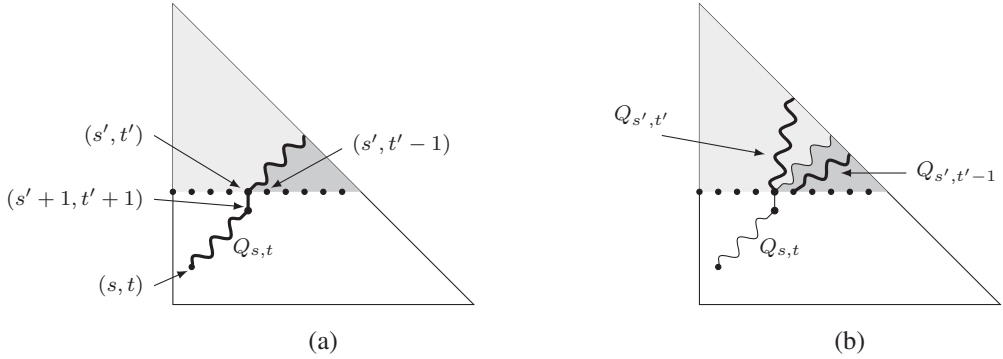

Recall that path $Q_{s,t}$ is encoded in row $t$ of $A_s$:  we create the path by reading this row \emph{in reverse}.  Next, observe that once we reach corner $(s',t')$ there are exactly $t'$ remaining steps in path $Q_{s,t}$.  These \emph{last} $t'$ steps of $Q_{s,t}$ are encoded by $A_s(t, [1:t'])$. 
Meanwhile, the path $Q_{s',t'} \in \mathcal{Q}_{s'}$ is encoded by row $t'$ of $A_{s'}$. Path $Q_{s',t'}$  must be northwest of the last $t'$ steps of $Q_{s,t}$, which corresponds to the algebraic constraint
$$
A_{s'}(t', [1:t']) \succeq A_{s}(t,[1:t']).
$$
 Similarly, the path $Q_{s',t'-1} \in \mathcal{Q}_{s'}$ must be southeast of last $t'$ steps of $Q_{s,t}$, which corresponds to
$$
A_{s'}(t'-1, [1:t'-1]) \prec A_{s}(t,[1:t']).
$$
All that remains is to observe that the cases $u=1$ and $u=2$ are exceptional. Having $A_s(t,1)=1$ means that the final step of $Q_{s,t}$ is a north step to point $(s',0)$ on the diagonal. This does not impact any of the paths in row $s'$, so no additional constraints are needed. Meanwhile, having $A_s(t,2)=1$ means that $t'=1$. There are no paths in $\mathcal{Q}_{s'}$ that are southeast of $Q_{s,t}$, so we do no need the second constraint in this case.
\end{proof}

\section{Conclusion}

\label{sec:end}

We have introduced approval ballot sequences into the ensemble of ballot problems. This family generalizes many known ballot problems, including ones counted by the Catalan, Schr\"{o}der, involution and switchboard numbers.
Algebraically, we represent these sequences with approval ballot triangles (ABTs). Geometrically, we visualize them as collections of noncrossing lattice paths (NCLPs). The latter vantage point reveals that ABTs are in bijection to totally symmetric  self-complementary plane partitions (TSSCPPs).

Using ABTs, we established a direct relationship between TSSCPPs and strict-sense ballots (SSBs).
The key was to introduce the notion of a hypertriangle: a list of triangular arrays 
in which sequential triangles adhere to a specified compatibility condition. We showed how to represent an SSB as an approval ballot hypertriangle, which is a list of compatible approval ballot triangles. 
Interestingly, this SSB decomposition parallels the structure of its constituent triangular arrays: each hypertriangle is made up of compatible TSSCPP triangles. Meanwhile, each ABT is made up of compatible rows of lattice paths. 
These are compelling examples of how the careful aggregation of  combinatorial families produces other important combinatorial families.
 This leads to a natural question: is there a known combinatorial family that can be constructed from a list of ``compatible SSBs?''

\bibliography{refs} 

\end{document}

%% file: preamble.tex
\usetikzlibrary{decorations.markings,cd}
\usetikzlibrary{positioning}
\usetikzlibrary{decorations.pathreplacing}

\newtheorem{theorem}{Theorem}

\newtheorem{prop}[theorem]{Proposition}
\newtheorem{proposition}[theorem]{Proposition}

\newtheorem*{conj*}{Conjecture}
\theoremstyle{definition}
\newtheorem{definition}[theorem]{Definition}

\newtheorem{remark}[theorem]{Remark}

\numberwithin{theorem}{section}
\numberwithin{equation}{section}

\newcommand{\abt}[1]{\mathcal{A}_{#1}}

\newcounter{x}
\newcounter{y}
\newcounter{z}

\newcommand\xaxis{210}
\newcommand\yaxis{-30}
\newcommand\zaxis{90}

\newcommand\topsidey[3]{
  \fill[fill=yellow!40, draw=black,shift={(\xaxis:#1)},shift={(\yaxis:#2)},
  shift={(\zaxis:#3)}] (0,0) -- (30:1) -- (0,1) --(150:1)--(0,0);
}

\newcommand\leftsider[3]{
  \fill[fill=red!40, draw=black,shift={(\xaxis:#1)},shift={(\yaxis:#2)},
  shift={(\zaxis:#3)}] (0,0) -- (0,-1) -- (210:1) --(150:1)--(0,0);
}

\newcommand\rightsideb[3]{
  \fill[fill=blue!40, draw=black,shift={(\xaxis:#1)},shift={(\yaxis:#2)},
  shift={(\zaxis:#3)}] (0,0) -- (30:1) -- (-30:1) --(0,-1)--(0,0);
}

\newcommand\topside[3]{
  \fill[fill=white, draw=black,shift={(\xaxis:#1)},shift={(\yaxis:#2)},
  shift={(\zaxis:#3)}] (0,0) -- (30:1) -- (0,1) --(150:1)--(0,0);
}

\newcommand\leftside[3]{
  \fill[fill=gray!10, draw=black,shift={(\xaxis:#1)},shift={(\yaxis:#2)},
  shift={(\zaxis:#3)}] (0,0) -- (0,-1) -- (210:1) --(150:1)--(0,0);
}

\newcommand\rightside[3]{
  \fill[fill=gray!20, draw=black,shift={(\xaxis:#1)},shift={(\yaxis:#2)},
  shift={(\zaxis:#3)}] (0,0) -- (30:1) -- (-30:1) --(0,-1)--(0,0);
}

\newcommand\topsidex[3]{
  \fill[fill=gray!30, draw=black,shift={(\xaxis:#1)},shift={(\yaxis:#2)},
  shift={(\zaxis:#3)}] (0,0) -- (30:1) -- (0,1) --(150:1)--(0,0);
}

\newcommand\leftsidex[3]{
  \fill[fill=gray!40, draw=black,shift={(\xaxis:#1)},shift={(\yaxis:#2)},
  shift={(\zaxis:#3)}] (0,0) -- (0,-1) -- (210:1) --(150:1)--(0,0);
}

\newcommand\rightsidex[3]{
  \fill[fill=gray!50, draw=black,shift={(\xaxis:#1)},shift={(\yaxis:#2)},
  shift={(\zaxis:#3)}] (0,0) -- (30:1) -- (-30:1) --(0,-1)--(0,0);
}

\newcommand\cube[3]{
  \topsidey{#1}{#2}{#3} \leftsider{#1}{#2}{#3} \rightsideb{#1}{#2}{#3}
}

\newcommand\cubex[3]{
  \topside{#1}{#2}{#3} \leftside{#1}{#2}{#3} \rightside{#1}{#2}{#3}
}

\newcommand\planepartition[1]{
 \setcounter{x}{-1}
  \foreach \a in {#1} {
    \addtocounter{x}{1}
    \setcounter{y}{-1}
    \foreach \b in \a {
      \addtocounter{y}{1}
      \setcounter{z}{-1}
      \foreach \c in {0,...,\b} {
        \addtocounter{z}{1}
      \ifthenelse{\c=0}{\setcounter{z}{-1},\addtocounter{y}{0}}{
        \cube{\value{x}}{\value{y}}{\value{z}}}
      }
    }
  }
}

\newcommand\planepartitionx[1]{
 \setcounter{x}{-1}
  \foreach \a in {#1} {
    \addtocounter{x}{1}
    \setcounter{y}{-1}
    \foreach \b in \a {
      \addtocounter{y}{1}
      \setcounter{z}{-1}
      \foreach \c in {0,...,\b} {
        \addtocounter{z}{1}
      \ifthenelse{\c=0}{\setcounter{z}{-1},\addtocounter{y}{0}}{
        \cubex{\value{x}}{\value{y}}{\value{z}}}
      }
    }
  }
}

\newcommand\upstep[4]{
  \fill[ultra thick,  draw=#4,shift={(\xaxis:#1)},shift={(\yaxis:#2)},
  shift={(\zaxis:#3)}] (0,0) -- (30:1);
}
\newcommand\flatstep[4]{
  \fill[ultra thick,  draw=#4,shift={(\xaxis:#1)},shift={(\yaxis:#2)},
  shift={(\zaxis:#3)}] (0,0) -- (150:1);
}

\newcommand\backstep[4]{
  \fill[ultra thick,  draw=#4,shift={(\xaxis:#1)},shift={(\yaxis:#2)},
  shift={(\zaxis:#3)}] (0,0) -- (270:1);
}

\newcommand\startdot[4]{
  \draw[thick, black, fill=white,shift={(\xaxis:#1)},shift={(\yaxis:#2)},
  shift={(\zaxis:#3)}] circle (5pt);
}

\newcommand{\bull}{\raise0.5pt \hbox{$\bullet$}}

\definecolor{codegreen}{rgb}{0,0.6,0}
\definecolor{codegray}{rgb}{0.5,0.5,0.5}
\definecolor{codepurple}{rgb}{0.58,0,0.82}
\definecolor{backcolour}{rgb}{0.95,0.95,0.92}

\lstdefinestyle{mystyle}{
    backgroundcolor=\color{backcolour},   
    commentstyle=\color{codegreen},
    keywordstyle=\color{magenta},
    numberstyle=\tiny\color{codegray},
    stringstyle=\color{codepurple},
    basicstyle=\ttfamily\footnotesize,
    breakatwhitespace=false,         
    breaklines=true,                 
    captionpos=b,                    
    keepspaces=true,                 
    numbers=left,                    
    numbersep=5pt,                  
    showspaces=false,                
    showstringspaces=false,
    showtabs=false,                  
    tabsize=2
}